\documentclass[11pt]{amsart}

\usepackage{amssymb, amscd}
\usepackage{epsfig, mathtools}
\usepackage[vmargin=1in, hmargin=1.25in]{geometry}
\usepackage[font=small,format=plain,labelfont=bf,up,textfont=it,up]{caption}
\usepackage{microtype}
\usepackage[shortlabels]{enumitem}
\usepackage[bookmarks, bookmarksdepth=2, colorlinks=true, linkcolor=blue, citecolor=blue, urlcolor=blue]{hyperref}

\setenumerate[0]{label=(\alph*)}

\newcommand{\arxiv}[1]{\href{http://arxiv.org/abs/#1}{{\tt arXiv:#1}}}

\numberwithin{equation}{section}

\theoremstyle{plain}
\newtheorem{theorem}{Theorem}[section]
\newtheorem{maintheorem}{Theorem}

\newtheorem{proposition}[theorem]{Proposition}
\newtheorem{lemma}[theorem]{Lemma}

\newtheorem{corollary}[theorem]{Corollary}

\newtheorem{question}[theorem]{Question}

\newtheorem{stepsa}{Step}
\newtheorem{stepsb}{Step}

\theoremstyle{definition}
\newtheorem{asm}[theorem]{Assumption}

\newtheorem{defn}[theorem]{Definition}
\newenvironment{definition}[1][]{\begin{defn}[#1]\pushQED{\qed}}{\popQED \end{defn}}
\newtheorem{notn}[theorem]{Notation}

\theoremstyle{remark}
\newtheorem{rmk}[theorem]{Remark}
\newenvironment{remark}[1][]{\begin{rmk}[#1] \pushQED{\qed}}{\popQED \end{rmk}}
\newtheorem{eg}[theorem]{Example}
\newenvironment{example}[1][]{\begin{eg}[#1] \pushQED{\qed}}{\popQED \end{eg}}
\newtheorem{conv}[theorem]{Convention}

\DeclareMathOperator{\Diff}{Diff}

\DeclareMathOperator{\SL}{SL}

\DeclareMathOperator{\Sp}{Sp}

\newcommand\Z{\ensuremath{\mathbb{Z}}}
\newcommand\Q{\ensuremath{\mathbb{Q}}}

\DeclareMathOperator{\HH}{H}

\newcommand\RH{\ensuremath{\widetilde{\HH}}}
\DeclareMathOperator{\CC}{C}

\DeclareMathOperator{\rank}{rk}

\DeclareMathOperator{\Aut}{Aut}
\DeclareMathOperator{\Out}{Out}

\newcommand\cB{\ensuremath{\mathcal{B}}}

\newcommand\cF{\ensuremath{\mathcal{F}}}

\newcommand\cL{\ensuremath{\mathcal{L}}}

\newcommand\cO{\ensuremath{\mathcal{O}}}

\newcommand\bA{\ensuremath{\mathbf{A}}}
\newcommand\bB{\ensuremath{\mathbf{B}}}

\newcommand\bbD{\ensuremath{\mathbb{D}}}

\newcommand\bbS{\ensuremath{\mathbb{S}}}

\newcommand\BTits{\ensuremath{\mathbf{Tits}  }}

\newcommand\hcL{\ensuremath{\widehat{\cL}}}

\DeclareMathOperator{\Tits}{Tits}
\DeclareMathOperator{\St}{St}
\DeclareMathOperator{\Mod}{Mod}
\DeclareMathOperator{\Curves}{\mathcal{C}}

\DeclareMathOperator{\Spheres}{\mathcal{S}}
\DeclareMathOperator{\SpheresNS}{\Spheres_{\text{ns}}}
\DeclareMathOperator{\NSpheres}{\mathcal{N}}
\DeclareMathOperator{\ANSpheres}{\mathcal{AN}}
\newcommand\ANSpheresEX{\ANSpheres_{\text{ex}}}
\DeclareMathOperator{\Link}{Link}
\DeclareMathOperator{\Cone}{Cone}

\DeclareMathOperator{\height}{ht}
\DeclareMathOperator{\Poset}{\mathcal{P}}
\DeclareMathOperator{\opp}{op}

\title[Free factors and dualizing modules]{The free factor complex and the dualizing module for the automorphism group of a free group}

\author{Zachary Himes}
\address{Dept of Mathematics; Purdue University; 150 N. University Street; West Lafayette, IN 47907}
\email{himesz@purdue.edu}

\author{Jeremy Miller}
\address{Dept of Mathematics; Purdue University; 150 N. University Street; West Lafayette, IN 47907}
\email{jeremykmiller@purdue.edu}
\thanks{JM was supported in part by a Simons Foundation Collaboration Grant.}

\author{Sam Nariman}
\address{Dept of Mathematics; Purdue University; 150 N. University Street; West Lafayette, IN 47907}
\email{snariman@purdue.edu}
\thanks{SN was supported in part by NSF grant DMS-2113828 and a Simons Foundation Collaboration Grant.}

\author{Andrew Putman}
\address{Dept of Mathematics; University of Notre Dame; 255 Hurley Hall; Notre Dame, IN 46556}
\email{andyp@nd.edu}
\thanks{AP was supported in part by NSF grant DMS-1811210.}

\date{November 7, 2022}

\begin{document}

\newpage

\begin{abstract}
Answering a question of Hatcher--Vogtmann, we prove that the top homology group of the free factor complex is not the dualizing module
for $\Aut(F_n)$, at least for $n = 5$.
\end{abstract}

\maketitle

\section{Introduction}
\label{section:introduction}

A group $\Gamma$ is a {\em rational duality group} of dimension $d$ if there exists a $\Q[\Gamma]$-module $\bbD$ called the
{\em dualizing module} such that for all $\Q[\Gamma]$-modules $M$, we have
\[\HH^{d-i}(G;M) \cong \HH_{i}(G;M \otimes_{\Q} \bbD) \quad \text{for all $i$}.\]
The rational cohomological dimension of such a $G$ is $d$, and the dualizing module $\bbD$ is unique:
\[\HH^d(G;\Q[G]) \cong \HH_0(G;\Q[G] \otimes_{\Q} \bbD) \cong \bbD.\]
Many geometrically important groups are rational duality groups,\footnote{What these references actually prove
is that these groups are virtual duality groups.  This is a stronger condition: all virtual duality groups
are rational duality groups, but the converse is false.  For instance, Deligne \cite{DeligneRF} constructed a
central extension $1 \rightarrow \mathbb{Z}/2 \rightarrow \widetilde{\Sp}_{2g}(\Z) \rightarrow \Sp_{2g}(\Z) \rightarrow 1$
such that $\widetilde{\Sp}_{2g}(\Z)$ has no torsion-free subgroup of finite index.  This implies that
$\widetilde{\Sp}_{2g}(\Z)$ is not a virtual duality group; however, since $\Sp_{2g}(\Z)$ and $\Z/2$ are rational
duality groups, the group $\widetilde{\Sp}_{2g}(\Z)$ is as well.}
e.g., lattices in semisimple Lie groups \cite{BorelSerre},
mapping class groups \cite{HarerDuality}, and outer automorphism groups of free groups \cite{BestvinaFeighnDuality}.

\subsection{Identifying the dualizing module}
The dualizing module of a rational duality group often has a simple geometric description.  Here are two examples:

\begin{example}
For $\SL_n(\Z)$, the associated {\em Tits building} is the geometric realization $\Tits(\Q^n)$ of
the poset of nontrivial proper subspaces of $\Q^n$.  The Solomon--Tits theorem \cite{SolomonTits} says that
$\Tits(\Q^n)$ is homotopy equivalent to a wedge of $(n-2)$-spheres.  The {\em Steinberg representation}
of $\SL_n(\Q)$, denoted $\St(\Q^n)$, is its unique nonzero reduced homology group:
\[\St(\Q^n) = \RH_{n-2}(\Tits(\Q^n);\Q).\]
Borel--Serre \cite{BorelSerre} proved $\SL_n(\Z)$ is a rational duality group with dualizing module $\St(\Q^n)$.
\end{example}

\begin{example}
For the mapping class group $\Mod(\Sigma_g)$ of a compact oriented genus-$g$ surface $\Sigma_g$, the {\em curve complex}
is the simplicial complex $\Curves_g$ whose $k$-simplices are sets $\{\gamma_0,\ldots,\gamma_k\}$ of distinct isotopy classes
of nontrivial simple closed curves on $\Sigma_g$ that can be realized disjointly.  Harer \cite{HarerDuality} proved
that for $g \geq 2$ the curve complex $\Curves_g$ is homotopy equivalent to a wedge of $(2g-2)$-spheres.  The
{\em Steinberg representation} of $\Mod(\Sigma_g)$, denoted $\St(\Sigma_g)$, is its unique nonzero reduced homology group:
\[\St(\Sigma_g) = \RH_{2g-2}(\Curves_g;\Q).\]
Harer \cite{HarerDuality} proved that $\Mod(\Sigma_g)$ is a rational duality group with dualizing module $\St(\Sigma_g)$.
\end{example}

These descriptions are useful since they allow calculations of the high-dimensional rational cohomology of
these groups and their finite-index subgroups; see, e.g.,
\cite{ChurchFarbPutmanVanish, ChurchFarbPutmanGL, ChurchPutmanCodim1, FullartonPutman, LeeSzczarba, MillerNagpalPatzt, MillerPatztPutman}.

\subsection{Steinberg module for automorphism group of free group}
Let $F_n$ be a free group of rank $n$.  Bestvina--Feighn \cite{BestvinaFeighnDuality} proved\footnote{What they actually
proved is that $\Out(F_n)$ is a rational duality group.  This fits into a short exact sequence
$1 \rightarrow F_n \rightarrow \Aut(F_n) \rightarrow \Out(F_n) \rightarrow 1$, and since both $\Out(F_n)$ and
$F_n$ are rational duality groups it follows that $\Aut(F_n)$ is as well.}
that $\Aut(F_n)$ is a rational
duality group of dimension $2n-2$.  However, their proof does not give an explicit model for its dualizing module.  Identifying
the dualizing module of $\Aut(F_n)$ remains a basic open problem.

Hatcher--Vogtmann \cite{HatcherVogtmann, HatcherVogtmannRevised} suggested studying the following.  A {\em free factor} of $F_n$
is a subgroup $A < F_n$ such that there exists another subgroup $B<F_n$ with $F_n = A \ast B$.  The
nontrivial proper free factors of $F_n$ form a poset called the {\em free factor complex}.
In analogy with the ordinary Tits building, we will denote its geometric realization by $\Tits(F_n)$, though
it is not actually a building.  Just like for $\Tits(\Q^n)$, Hatcher--Vogtmann \cite{HatcherVogtmann, HatcherVogtmannRevised}
proved that $\Tits(F_n)$ is homotopy equivalent to a wedge of $(n-2)$-dimensional spheres.
The {\em Steinberg module} for $\Aut(F_n)$, denoted $\St(F_n)$, is its unique nonzero reduced
homology group:
\[\St(F_n) = \RH_{n-2}(\Tits(F_n);\Q).\]
On \cite[p.\ 1]{HatcherVogtmann, HatcherVogtmannRevised}, Hatcher-Vogtmann asked the following question:

\begin{question}[Hatcher--Vogtmann]
\label{question:hatchervogtmann}
Is $\St(F_n)$ the dualizing module for $\Aut(F_n)$?
\end{question}

A consequence of our main theorem is that this question has a negative answer in general.

\subsection{Main theorem}
Our main theorem is as follows:

\begin{maintheorem}
\label{maintheorem:steinberghomology}
$\HH_i(\Aut(F_n);\St(F_n)) = 0$ for $n \geq 2$ and $i=0$ or $1$.
\end{maintheorem}

If $\St(F_n)$ were the dualizing module for $\Aut(F_n)$, this would imply that
\[\HH^{2n-2}(\Aut(F_n);\Q) \cong \HH_0(\Aut(F_n);\St(F_n)) = 0 \quad \text{for $n \geq 2$}\]
and
\[\HH^{2n-3}(\Aut(F_n);\Q) \cong \HH_1(\Aut(F_n);\St(F_n)) = 0 \quad \text{for $n \geq 3$}.\]
However, Gerlits \cite{GerlitsThesis} used a computer\footnote{See \cite{GerlitsCode} for the code.} to prove that
$\HH^7(\Aut(F_5);\Q) \cong \Q$, contradicting the second assertion.  We deduce the following:

\begin{corollary}
The Steinberg module $\St(F_5)$ is not the dualizing module for $\Aut(F_5)$.
\end{corollary}

\begin{remark}
We do not know all the rational cohomology of $\Aut(F_n)$ for any $n \geq 6$, so it is
unclear whether or not our theorem implies that $\St(F_n)$ is not the dualizing module
for $\Aut(F_n)$ for $n \geq 6$.  In any case, we conjecture that it is never the dualizing
module except possibly in some low-rank degenerate cases.
\end{remark}

\begin{remark}
In \cite{BruckGupta}, Br\"{u}ck--Gupta study an $\Out(F_n)$-variant of $\Tits(F_n)$.
It would be interesting to adapt our techniques to study this complex.
\end{remark}

\subsection{Presentation}
To prove Theorem \ref{maintheorem:steinberghomology}, we construct an explicit presentation
for the $\Q[\Aut(F_n)]$-module $\St(F_n)$.  Our inspiration is a beautiful presentation for
$\St(\Q^n)$ constructed by Bykovskii \cite{Bykovskii}.  Church--Putman \cite{ChurchPutmanCodim1}
gave an alternate topological proof of Bykovskii's theorem, and we adapt their
approach to $\Aut(F_n)$.  The key is to find a highly connected simplicial complex that
forms an ``integral model'' for
the free factor complex $\Tits(F_n)$.  The simplicial complex we use is a variant of one used by Hatcher--Vogtmann's \cite{HatcherVogtmann, HatcherVogtmannRevised}
to prove that $\Tits(F_n)$ is homotopy equivalent to a
wedge of $(n-2)$-spheres.

\subsection{Sphere complex}
To describe this simplicial complex, we need to make two definitions.

\begin{definition}
Let $M_{n,b}$ be the connect sum of $n$ copies of $S^2 \times S^1$ with
$b$ disjoint open balls removed.  Our convention is that the connect sum of
$0$ copies of $S^2 \times S^1$ is the unit $S^3$ for the connect sum operation,
so $M_{0,b}$ is $S^3$ with $b$ disjoint open balls removed.
\end{definition}

\begin{definition}
A $2$-sphere embedded in $M_{n,b}$ is {\em essential} if it is not homotopic to $\partial M_{n,b}$ or a point.
A {\em rank-$k$ sphere system} in $M_{n,b}$ is a set $\{S_0,\ldots,S_k\}$ of distinct isotopy classes of
essential $2$-spheres embedded in $M_{n,b}$ that
can be realized disjointly.
\end{definition}

We will systematically conflate $2$-spheres in $M_{n,b}$ with their isotopy classes.  We
can now define the key simplicial complex, which was originally introduced by Hatcher \cite{HatcherStabilization}.

\begin{definition}
The {\em sphere complex} for $M_{n,b}$,
denoted $\Spheres(M_{n,b})$, is the simplicial complex whose $k$-simplices
are rank-$k$ sphere systems $\{S_0,\ldots,S_k\}$ in $M_{n,b}$.
\end{definition}

To connect this to $\Aut(F_n)$, fix a basepoint $\ast \in \partial M_{n,1}$.  The mapping class group
\[\Mod(M_{n,1}) = \pi_0(\Diff^{+}(M_{n,1},\partial M_{n,1}))\]
acts on $\pi_1(M_{n,1},\ast) \cong F_n$, giving a homomorphism
$\Mod(M_{n,1}) \rightarrow \Aut(F_n)$.  Laudenbach \cite{LaudenbachPaper, LaudenbachBook} proved
that this homomorphism is surjective and that its kernel is a rank-$n$ abelian $2$-group
generated by ``sphere twists'':
\[1 \longrightarrow (\Z/2)^n \longrightarrow \Mod(M_{n,1}) \longrightarrow \Aut(F_n) \longrightarrow 1.\]
See \cite{BrendleBroaddusPutman} for an alternate proof that shows that this
exact sequence splits.  Laudenbach also proved that all elements of the kernel fix the isotopy classes of all $2$-spheres
in $M_{n,1}$.  It follows that the action of $\Mod(M_{n,1})$ on $\Spheres(M_{n,1})$ factors
through an action of $\Aut(F_n)$ on $\Spheres(M_{n,1})$.

\subsection{Hatcher--Vogtmann's proof}
The starting point of Hatcher--Vogtmann's proof that $\Tits(F_n)$ is homotopy equivalent to a
wedge of $(n-2)$-spheres is a theorem of
Hatcher \cite{HatcherStabilization} saying that $\Spheres(M_{n,1})$ is contractible.  Hatcher--Vogtmann
consider the following subcomplex:

\begin{definition}
\label{definition:nspheres}
The {\em nonseparating sphere complex}\footnote{This is different from the complex of nonseparating spheres,
which we discuss in Definition \ref{definition:nonseparatingspheres}.  In the complex of nonseparating spheres,
the vertices are required to be nonseparating spheres, but there is no condition on the higher-dimensional
simplices.  The nonseparating sphere complex is also sometimes called the {\em purely} nonseparating
sphere complex.} of $M_{n,b}$, denoted $\NSpheres(M_{n,b})$, is the subcomplex
of $\Spheres(M_{n,b})$ whose $k$-simplices are rank-$k$ sphere systems $\{S_0,\ldots,S_k\}$ such that the union of the $S_i$
does not separate $M_{n,b}$.
\end{definition}

Using the fact that $\Spheres(M_{n,1})$ is contractible, Hatcher--Vogtmann prove that $\NSpheres(M_{n,1})$ is
homotopy equivalent to a wedge of $(n-1)$-spheres.  For each simplex $\{S_0,\ldots,S_k\}$
of $\NSpheres(M_{n,1})$, the fundamental group of the complement
\[\pi_1(M_{n,1} \setminus \bigcup_{i=0}^k S_i,\ast)\]
is a proper free factor of $F_n$.  It is a nontrivial free factor precisely when the simplex
is not a maximal simplex.  This provides a bridge between $\NSpheres(M_{n,1})$ and $\Tits(F_n)$,
which Hatcher--Vogtmann use to prove that $\Tits(F_n)$ is homotopy equivalent to a
wedge of $(n-2)$-spheres.

\subsection{Our approach}
It turns out that the fact that $\NSpheres(M_{n,1})$ is
homotopy equivalent to a wedge of $(n-1)$-spheres is sufficient to construct generators
for $\St(F_n)$, which are a sort of free group version of the spherical apartments in
a Tits building.  Following Church--Putman \cite{ChurchPutmanCodim1}, to find
the relations between these generators we need to add simplices to $\NSpheres(M_{n,1})$
to increase its connectivity from $(n-1)$ to $n$.  The result is what we call
the ``augmented nonseparating sphere complex'', and most of this paper is
a detailed study of its topology.

\begin{remark}
Though he did not state it explicitly and it takes some work to extract it from his paper,
in \cite{CostaBases} Sadofschi Costa proved results that are equivalent to our generating set for $\St(F_n)$.
\end{remark}

\subsection{Outline}
There are three main parts to this paper.  The first (\S \ref{section:prelim} -- \S \ref{section:badsimplex})
is several sections of topological preliminaries.  The second
(\S \ref{section:spheres} -- \S \ref{section:augmenteddisc}) is a detailed study of
the sphere complex and its subcomplexes.  Finally, the third (\S \ref{section:presentation} -- \S \ref{section:mainproof}) contains the proof of Theorem \ref{maintheorem:steinberghomology}.

\subsection{Acknowledgments}
We thank Benjamin Br\"{u}ck and Peter Patzt for helpful conversations and Richard Wade for comments
on a previous version of this paper.  We also would like to thank the referee for their careful
reading of our paper.

\section{Topological preliminaries}
\label{section:prelim}

We begin with some generalities about the topology of simplicial complexes.

\subsection{Connectivity conventions}
\label{section:connectivityconventions}
For any $d \in \Z$, we say that a space $X$ is $d$-connected if for $k \leq d$, all maps $\bbS^{k} \rightarrow X$
from the $k$-sphere $\bbS^{k}$ to $X$ extend over the $(k+1)$-disc $\bbD^{k+1}$.  There are two important
edge cases to this convention:
\begin{itemize}
\item For $k \leq -2$, we have $\bbS^k = \bbD^{k+1} = \emptyset$, so all spaces are $d$-connected for $d \leq -2$.
\item We have $\bbS^{-1} = \emptyset$ and $\bbD^0 = \{\text{pt}\}$, so a space $X$ is $(-1)$-connected precisely
when $X \neq \emptyset$.
\end{itemize}

\subsection{Links}
For a simplex $\sigma$ of a simplicial complex $X$, the {\em link} of $\sigma$, denoted $\Link_X(\sigma)$, is
the subcomplex of $X$ consisting of all simplices $\sigma'$ such that the join $\sigma \ast \sigma'$
is a simplex of $X$.  Notice that this implies that $\sigma'$ does not share any vertices with $\sigma$.
The join $\sigma \ast \Link_X(\sigma)$ is a subcomplex of $X$ called the {\em star} of $\sigma$.

\subsection{Combinatorial triangulations}
A {\em combinatorial triangulation} of an $n$-manifold with boundary
is defined inductively as follows.  First, 
any $0$-dimensional simplicial complex is a combinatorial triangulation of a $0$-manifold. 
 Next, for $n \geq 1$ a combinatorial triangulation of an $n$-manifold with boundary 
is a simplicial complex $M^n$ that is a topological $n$-manifold with boundary such that for all
simplices $\sigma$ of $M^n$, the complex $\Link_{M^n}(\sigma)$ is as follows:
\begin{itemize}
\item If $\sigma$ does not lie in $\partial M^n$, then $\Link_{M^n}(\sigma)$ is a combinatorial
triangulation of an $(n-\dim(\sigma)-1)$-sphere.
\item If $\sigma$ lies in $\partial M^n$, then $\Link_{M^n}(\sigma)$ is a combinatorial
triangulation of an $(n-\dim(\sigma)-1)$-disc.
\end{itemize}
Any subdivision of a combinatorial triangulation of a manifold with boundary
is also a combinatorial triangulation.

\section{Locally injective maps, spheres, and discs}
\label{section:localinj}

We now turn to some more technical aspects of simplicial complexes.

\subsection{Local injectivity}
If $f\colon Y \rightarrow X$ is a simplicial map between simplicial complexes, then for all
simplices $\sigma$ of $Y$ we have $\dim(\sigma) \geq \dim(f(\sigma))$.  For technical reasons,
it will be important for us to make this an equality whenever possible:

\begin{definition}
A simplicial map $f\colon Y \rightarrow X$ between simplicial complexes is {\em locally injective}
if for all simplices $\sigma$ of $Y$, we have $\dim(\sigma) = \dim(f(\sigma))$.
\end{definition}

\subsection{Sphere and disc local injectivity properties}
To keep our homotopies from getting out of control, it will be helpful to represent
homotopy classes in simplicial complexes by locally injective maps.  We therefore
make the following two definitions.

\begin{definition}
\label{definition:spherelocalinj}
We say that a simplicial complex $X$ has the {\em sphere local injectivity property} up to
dimension $d$ if the following holds for all $k \leq d$.  Let $f\colon \bbS^k \rightarrow X$
be a continuous map.  Then $f$ can be homotoped to a map $f\colon \bbS^k \rightarrow X$ that
is simplicial and locally injective for some combinatorial triangulation of $\bbS^k$.
\end{definition}

\begin{definition}
\label{definition:disclocalinj}
We say that a simplicial complex $X$ has the {\em disc local injectivity property} up to
dimension $d$ if the following holds for all $k \leq d$.  Let $f\colon \bbS^k \rightarrow X$
be a map that is simplicial and locally injective for some combinatorial triangulation
of $\bbS^k$.  Then there exists a combinatorial triangulation of $\bbD^{k+1}$ that extends
our given triangulation of $\bbS^k$ and a locally injective simplicial map $F\colon \bbD^{k+1} \rightarrow X$
extending $f$.
\end{definition}

\begin{example}
\label{example:edgecases}
All simplicial complexes have the sphere local injectivity property up to dimension $0$.
A simplicial complex has the disc local injectivity property up to dimension $-1$ if it is nonempty,
and has the disc local injectivity property up to dimension $0$ if it is nonempty, connected, and not
just a single point.
\end{example}

\begin{example}
It will follow from our discussion of weakly Cohen--Macaulay complexes in \S \ref{section:cohenmac}
that if $M^n$ is a combinatorial triangulation of an $n$-manifold that is $(d-1)$-connected for some $d \leq n$, then $M^n$
has the sphere local injectivity property up to dimension $d$ and disc local injectivity property 
up to dimension $d-1$.  See Example \ref{example:manifoldcm} and Lemma \ref{lemma:makeinjectivecm}.
\end{example}

\begin{remark}
\label{remark:ADD1}
By the simplicial approximation theorem, any continuous map $f\colon \bbS^k \rightarrow X$ is
homotopic to a map that is simplicial with respect to a triangulation of $\bbS^k$ that
can be obtained by subdividing any given triangulation of $\bbS^k$.  Since the class
of combinatorial triangulations of a manifold is closed under subdivisions, by starting with
a combinatorial triangulation of $\bbS^k$ (for instance, the boundary of a $(k+1)$-simplex)
we can ensure that the resulting map $\bbS^k \rightarrow X$ is simplicial with respect to
a combinatorial triangulation of $\bbS^k$.  Thus the local injectivity part of the sphere
local injectivity property is the key content of that definition.  A similar remark applies
to the disc local injectivity property, though here you need Zeeman's version \cite{ZeemanRelative} of the
relative simplicial approximation theorem.\footnote{There are subtle issues with relative
simplicial approximation, and the standard version as found in e.g.\ \cite[Theorem 3.4.8]{Spanier}
is not strong enough for what we are doing.  See the introduction to \cite{ZeemanRelative}
for a discussion of this.}
\end{remark}

We have the following.

\begin{lemma}
\label{lemma:localinjcon}
Let $X$ be a simplicial complex that has both the sphere and disc local injectivity property
up to dimension $d$.  Then $X$ is $d$-connected.
\end{lemma}
\begin{proof}
For $k \leq d$, first use the sphere local injectivity property to homotope a given
$f\colon \bbS^k \rightarrow X$ to a locally injective map, and then use the disc local injectivity
property to extend $f$ over $\bbD^{k+1}$.
\end{proof}

\subsection{Joins}
The sphere and disc local injectivity properties behave well under joins.  We will need this for the
disc local injectivity property, so we focus on that:

\begin{lemma}
\label{lemma:discjoin}
Let $X$ be a simplicial complex with the disc local injectivity property up to dimension $d$ and $Y$
be a simplicial complex with the disc local injectivity property up to dimension $e$.  Then $X \ast Y$
has the disc local injectivity property up to dimension $d+e+2$.
\end{lemma}
\begin{proof}
For some $k \leq d+e+2$, fix a combinatorial triangulation of $\bbS^k$ and let $f\colon \bbS^k \rightarrow X \ast Y$
be a locally injective simplicial map.  Our goal is to extend $f$ to a locally injective map
$F\colon \bbD^{k+1} \rightarrow X \ast Y$ for some combinatorial triangulation of $\bbD^{k+1}$ agreeing
with our triangulation of $\bbS^k$ on the boundary.  We divide the proof of this into three steps.

\begin{stepsa}
We modify $f$ such that $\dim(\sigma) \leq d$ for all simplices $\sigma$ of $\bbS^k$ with $f(\sigma) \subset X$.
\end{stepsa}

Consider a simplex $\sigma$ of $\bbS^k$ with $f(\sigma) \subset X$.  Pick $\sigma$ such that its dimension is maximal
with this property, and assume that $\dim(\sigma) \geq d+1$.  Then $\Link_{\bbS^k}(\sigma)$ is a combinatorial
triangulation of a $(k-\dim(\sigma)-1)$ sphere, and by the maximality of the dimension of $\sigma$ we have
$f(\Link_{\bbS^k}(\sigma)) \subset Y$.  Note that
\[k-\dim(\sigma)-1 \leq (d+e+2) - (d+1) -1 = e.\]
Since $Y$ has the disc local injectivity property up to dimension $e$ and $f$ is locally injective, we can
find a combinatorial triangulation of $\bbD^{k-\dim(\sigma)}$ agreeing with that of $\Link_{\bbS^k}(\sigma) \cong \bbS^{k-\dim(\sigma)-1}$ on
$\partial \bbD^{k-\dim(\sigma)}$ and a locally injective map $G\colon \bbD^{k-\dim(\sigma)} \rightarrow Y$
extending the restriction of $f$ to $\Link_{\bbS^k}(\sigma)$.  We have
\[\sigma \ast \Link_{\bbS^k}\left(\sigma\right) \cong \bbD^{\dim(\sigma)} \ast \bbS^{k-\dim(\sigma)-1} \cong \bbD^k\]
and
\[\partial \sigma \ast \bbD^{k-\dim(\sigma)} \cong \bbS^{\dim(\sigma)-1} \ast \bbD^{k-\dim(\sigma)} \cong \bbD^k,\]
and also
\[\partial\left(\sigma \ast \Link_{\bbS^k}\left(\sigma\right)\right) = \partial \sigma \ast \Link_{\bbS^k}\left(\sigma\right)\]
and
\[\partial\left(\partial \sigma \ast \bbD^{k-\dim(\sigma)}\right) = \partial \sigma \ast \partial \bbD^{k-\dim(\sigma)} = \partial \sigma \ast \Link_{\bbS^k}(\sigma).\]
In all of these $=$ means equality of simplicial complexes and $\cong$ means homeomorphism.  It follows that we can homotope
$f$ so as to replace 
\[f|_{\sigma \ast \Link_{\bbS^k}(\sigma)}\colon \sigma \ast \Link_{\bbS^k}(\sigma) \rightarrow X \ast Y \quad \text{with} \quad \left(f|_{\partial \sigma}\right) \ast G \colon \partial \sigma \ast \bbD^{k-\dim(\sigma)} \rightarrow X \ast Y.\]
This eliminates $\sigma$, and it is easy to see that it is enough to prove the lemma for this new $f$.
Repeating this over and over, we can ensure that $\dim(\sigma) \leq d$ for all simplices $\sigma$ of $\bbS^k$ with $f(\sigma) \subset X$.

\begin{stepsa}
For a topological space $W$, let $\Cone(W)$ denote the cone on $W$.  
Set $Z = f^{-1}(X) \subset \bbS^k$.  We construct a triangulation of $\Cone(Z)$ and a locally injective
simplicial map $F\colon \Cone(Z) \rightarrow X$ extending $f|_Z$.
\end{stepsa}

Let $p_0$ be the cone point of $\Cone(Z)$, and define $F(p_0)$ to be some arbitrary vertex of $X$.  By the previous
step, $Z$ has dimension at most $d$.  We can now use the disc local injectivity property of $X$ up to dimension $d$
to extend $F$ over $\sigma \ast p_0$ for each simplex $\sigma$ of $Z$: first over the $0$-simplices, then the
$1$-simplices, etc.  At each step, we have already defined $F$ on a subdivision of
\[\sigma \cup_{\partial \sigma} \left(\partial \sigma \ast p_0\right) \cong \bbD^{\dim(\sigma)} \cup_{\bbS^{\dim(\sigma)-1}} \left(\bbS^{\dim(\sigma)-1} \ast p_0\right) \cong \bbS^{\dim(\sigma)},\]
and we use the disc local injectivity property to extend $F$ to a locally injective map on the interior of this simplex (after a
further subdivision).

\begin{stepsa}
We extend $F$ to the rest of $\bbD^{k+1}$.
\end{stepsa}

We have defined $F$ on a subdivision of $\Cone(Z) \subset \Cone(\bbS^k)$, and $F$ takes $\Cone(Z)$ to $X$.  Subdivide
$\Cone(\bbS^k)$ by subdividing exactly the same simplices subdivided to form our subdivision of $\Cone(Z)$ on
which $F$ was defined.  The simplices of this subdivision are thus all of the form $\sigma \ast \tau$, where
$\sigma$ is a simplex of $\Cone(Z)$ with $F(\sigma) \subset X$ and $\tau$ is a simplex of $\bbS^k$ with $f(\tau) \subset Y$.
From this, we see that $F$ can be extended over all these simplices to a map with values in $X \ast Y$, as desired.
\end{proof}

\section{Cohen--Macaulay complexes}
\label{section:cohenmac}

We now discuss an important class of simplicial complexes.

\subsection{Weakly Cohen--Macaulay complexes}
We start with the following definition:

\begin{definition}
\label{definition:cm}
Let $X$ be a simplicial complex and let $d \geq 0$.
We say that $X$ is {\em weakly Cohen--Macaulay} of dimension $d$ if $X$ is $(d-1)$-connected,
and for all simplices $\sigma$ of $X$ the subcomplex $\Link_X(\sigma)$ is 
$(d-\dim(\sigma)-2)$-connected.\footnote{If we considered the empty set to be a $(-1)$-dimensional
simplex whose link is the whole complex $X$, then we could combine the two hypotheses and just
say that for all simplices $\sigma$ of $X$ the subcomplex $\Link_X(\sigma)$ is 
$(d-\dim(\sigma)-2)$-connected.}
If in addition to this $X$ is $d$-dimensional and for all simplices $\sigma$ of $X$ the subcomplex
$\Link_X(\sigma)$ is $(d-\dim(\sigma)-1)$-dimensional, then we say that $X$ is {\em Cohen--Macaulay} of dimension $d$.
\end{definition}

\begin{remark}
If $X$ is weakly Cohen--Macaulay of dimension $d$ and $\sigma$ is a simplex of $X$, then
$\Link_X(\sigma)$ is weakly Cohen--Macaulay of dimension $(d-\dim(\sigma)-1)$.  A similar
remark applies if $X$ is Cohen--Macaulay of dimension $d$.
\end{remark}

\begin{example}
\label{example:manifoldcm}
Let $M^n$ be a combinatorial triangulation of an $n$-dimensional manifold that is $(d-1)$-connected for some
$d \leq n$.  Then $M^n$ is weakly Cohen--Macaulay of dimension $d$.  Indeed, it is $(d-1)$-connected
by assumption, and for a simplex $\sigma$ of $M^n$ we have that $\Link_{M^n}(\sigma)$ is a
combinatorial triangulation of an $(n-\dim(\sigma)-1)$-sphere, and in particular is
\[n-\dim(\sigma)-2 \geq d-\dim(\sigma) -2\]
connected.
\end{example}

\subsection{Cohen--Macaulay and local injectivity}
Weakly Cohen--Macaulay complexes provide examples of the sphere and disc local injectivity property:

\begin{lemma}\hspace{-5pt}\footnote{See Galatius--Randal-Williams \cite[Theorem 2.4]{GRW} for a similar result.}
\label{lemma:makeinjectivecm}
Let $X$ be a simplicial complex that is weakly Cohen--Macaulay of dimension $d$.  Then
$X$ has the sphere local injectivity property up to dimension $d$ and the disc local injectivity
property up to dimension $(d-1)$.
\end{lemma}
\begin{proof}
The proof will be by induction on $d$.  The base case $d=0$ is trivial,\footnote{See Remark \ref{example:edgecases}, and
note that a simplicial complex $X$ that is weakly Cohen--Macaulay of dimension $0$ is $(-1)$-connected, i.e., nonempty.} so assume that $d>0$ and that
the lemma is true whenever $d$ is smaller.  We prove the two parts of the lemma separately.

\begin{stepsb}
The simplicial complex $X$ has the sphere local injectivity property up to dimension $d$.
\end{stepsb}

For some $k \leq d$, let $f\colon \bbS^k \rightarrow X$ be a continuous map.  Using simplicial
approximation (see Remark \ref{remark:ADD1}), 
we can assume that $f$ is simplicial for some combinatorial triangulation of
$\bbS^k$.  We will prove that $f$ can be homotoped to a locally injective simplicial map.

Assume that $f$ is not locally injective.
Consider a simplex $\sigma$ of $\bbS^k$ with $\dim(f(\sigma)) < \dim(\sigma)$.  Pick $\sigma$ such that its dimension is maximal
with this property.  Then $\Link_{\bbS^k}(\sigma)$ is a combinatorial
triangulation of a $(k-\dim(\sigma)-1)$ sphere, and by the maximality of the dimension of $\sigma$ we have
$f(\Link_{\bbS^k}(\sigma)) \subset \Link_X(f(\sigma))$.  This maximality also implies that the restriction
of $f$ to $\Link_{\bbS^k}(\sigma)$ is locally injective.  Note that
\[k-\dim(\sigma)-1 \leq d - \dim(f(\sigma)) - 2.\]
Since $X$ is weakly Cohen--Macaulay of dimension $d$, it follows that $\Link_X(f(\sigma))$ is weakly
Cohen--Macaulay of dimension $d-\dim(f(\sigma))-1$.  In particular, by our inductive hypothesis
it has the disc local injectivity property up to dimension $d-\dim(f(\sigma))-2$.  We can thus
find a combinatorial triangulation of $\bbD^{k-\dim(\sigma)}$ agreeing with that of $\Link_{\bbS^k}(\sigma) \cong \bbS^{k-\dim(\sigma)-1}$ on
$\partial \bbD^{k-\dim(\sigma)}$ and a locally injective map $G\colon \bbD^{k-\dim(\sigma)} \rightarrow X$
extending the restriction of $f$ to $\Link_{\bbS^k}(\sigma)$.  We have
\[\sigma \ast \Link_{\bbS^k}\left(\sigma\right) \cong \bbD^{\dim(\sigma)} \ast \bbS^{k-\dim(\sigma)-1} \cong \bbD^k\]
and
\[\partial \sigma \ast \bbD^{k-\dim(\sigma)} \cong \bbS^{\dim(\sigma)-1} \ast \bbD^{k-\dim(\sigma)} \cong \bbD^k,\]
and also
\[\partial\left(\sigma \ast \Link_{\bbS^k}\left(\sigma\right)\right) = \partial \sigma \ast \Link_{\bbS^k}\left(\sigma\right)\]
and
\[\partial\left(\partial \sigma \ast \bbD^{k-\dim(\sigma)}\right) = \partial \sigma \ast \partial \bbD^{k-\dim(\sigma)} = \partial \sigma \ast \Link_{\bbS^k}(\sigma).\]
In all of these $=$ means equality of simplicial complexes and $\cong$ means homeomorphism.  It follows that we can homotope
$f$ so as to replace
\[f|_{\sigma \ast \Link_{\bbS^k}(\sigma)} \colon \sigma \ast \Link_{\bbS^k}(\sigma) \rightarrow X\]
with
\[\left(f|_{\partial \sigma}\right) \ast G \colon \partial \sigma \ast \bbD^{k-\dim(\sigma)} \rightarrow f(\sigma) \ast \Link_X(f(\sigma)) \subset X.\]
This eliminates $\sigma$. 
Repeating this over and over, we can ensure that $f$ is locally injective, as desired.

\begin{stepsb}
The simplicial complex $X$ has the disc local injectivity property up to dimension $d-1$.
\end{stepsb}

For some $k \leq d-1$, let $f\colon \bbS^k \rightarrow X$ be a map that is locally injective with respect
to some combinatorial triangulation of $\bbS^k$.  Since $X$ is weakly Cohen--Macaulay of dimension $d$, it
is $(d-1)$-connected.  It follows (see Remark \ref{remark:ADD1}) 
that we can extend $f$ to a map $F\colon \bbD^{k+1} \rightarrow X$ that
is simplicial with respect to a combinatorial triangulation of $\bbD^{k+1}$.  Using an argument identical
to the previous step, we can modify $F$ without changing it on the boundary to ensure that it is locally
injective, as desired.
\end{proof}

\section{Bad simplex arguments}
\label{section:badsimplex}

Let $Y$ be a simplicial complex and let $X \subset Y$ be a subcomplex.  In this section, we discuss
two results that let us relate the topological properties of $X$ and $Y$ by studying a set
of ``bad simplices'' that characterize simplices that are in $Y$ but not in $X$.

\subsection{Connectivity and bad simplices}
The first of the two results that we will need is due to Hatcher--Vogtmann \cite{HatcherVogtmannTethers}.  To
state it, we must first give a definition.

\begin{definition}
\label{definition:badlink}
Let $Y$ be a simplicial complex and let $\cB$ be a set of simplices of $Y$.  For $\sigma \in \cB$,
let $\cL(\sigma,\cB)$ be the subcomplex of $Y$ consisting of simplices $\tau$ satisfying the following
two conditions:
\begin{itemize}
\item $\tau$ is in $\Link_Y(\sigma)$, so $\sigma \ast \tau$ is a simplex of $Y$.
\item If $\sigma'$ is a face of $\sigma \ast \tau$ with $\sigma' \in \cB$, then $\sigma'$
is a face of $\sigma$.\qedhere
\end{itemize}
\end{definition}

We can now state Hatcher--Vogtmann's theorem as follows.  In it, you should regard $\cB$ as
the set of ``bad simplices'' of $Y$:

\begin{proposition}[{\cite[Proposition 2.1]{HatcherVogtmannTethers}}]
\label{proposition:avoidbadeasy}
Let $Y$ be a simplicial complex and let $X \subset Y$ be a subcomplex.  Assume that there
exists a set $\cB$ of simplices of $Y$ with the following properties for some $d \geq 0$:
\begin{itemize}
\item[(i)] A simplex of $Y$ lies in $X$ if and only if none of its faces are in $\cB$.  In particular,
since simplices are faces of themselves no simplices of $\cB$ lie in $X$. 
\item[(ii)] If $\sigma,\sigma' \in \cB$ are such that $\sigma \cup \sigma'$ is a simplex
of $Y$, then $\sigma \cup \sigma' \in \cB$.  Note that $\sigma$ and $\sigma'$ might share vertices,
so $\sigma \cup \sigma'$ might not be $\sigma \ast \sigma'$.
\item[(iii)] For all $\sigma \in \cB$, the subcomplex $\cL(\sigma,\cB)$ is $(d-\dim(\sigma)-1)$-connected.
\end{itemize}
Then the pair $(Y,X)$ is $d$-connected, i.e., we have $\pi_k(Y,X)=0$ for $0 \leq k \leq d$.
\end{proposition}

Both to make this paper more self-contained and to motivate the second and more technical bad
simplex argument we discuss below, we include a proof.

\begin{proof}[Proof of Proposition \ref{proposition:avoidbadeasy}]
For some $0 \leq k \leq d$, let $f\colon (\bbD^k,\partial \bbD^k) \rightarrow (Y,X)$ 
be a map of pairs that is simplicial with respect to some combinatorial triangulation
of $\bbD^k$.  Our goal is to homotope $f$ rel boundary such that its image lies in $X$.
So assume that the image of $f$ does not lie in $X$.

By (i), there exists a simplex $\sigma$ of $\bbD^k$ such that $f(\sigma) \in \cB$.  Pick 
$\sigma$ such that its dimension is maximal among all simplices with $f(\sigma) \in \cB$.  
Consider a simplex $\tau$ of $\Link_{\bbD^k}(\sigma)$.  We claim that $f(\tau) \in \cL(f(\sigma),\cB)$.
Indeed, the maximality of $\dim(\sigma)$ implies that $f(\tau)$ is in $\Link_Y(f(\sigma))$.  If
$f(\tau)$ does not lie in $\cL(f(\sigma),\cB)$, then $f(\sigma) \ast f(\tau)$ contains a face
$\eta$ with $\eta \in \cB$ and with $\eta$ containing vertices of $f(\tau)$.  But then
(ii) implies that $f(\sigma) \cup \eta \in \cB$, contradicting the maximality of
$\dim(\sigma)$.

Since $f$ takes $\partial \bbD^k$ to $X$, the simplex $\sigma$ does not lie in $\partial \bbD^k$.
It follows that $\Link_{\bbD^k}(\sigma)$ is a combinatorial
triangulation of a $(k-\dim(\sigma)-1)$ sphere, and by 
(iii) the complex $\cL(f(\sigma),\cB)$ is $(d-\dim(\sigma)-1)$-connected.  Since
\[k-\dim(\sigma)-1 \leq d-\dim(\sigma)-1,\]
we can find a combinatorial triangulation of $\bbD^{k-\dim(\sigma)}$ agreeing with that of 
$\Link_{\bbD^k}(\sigma) \cong \bbS^{k-\dim(\sigma)-1}$ on
$\partial \bbD^{k-\dim(\sigma)}$ and a simplicial map $g\colon \bbD^{k-\dim(\sigma)} \rightarrow \cL(f(\sigma),\cB)$
extending the restriction of $f$ to $\Link_{\bbD^k}(\sigma)$.  We have
\[\sigma \ast \Link_{\bbD^k}\left(\sigma\right) \cong \bbD^{\dim(\sigma)} \ast \bbS^{k-\dim(\sigma)-1} \cong \bbD^k\]
and
\[\partial \sigma \ast \bbD^{k-\dim(\sigma)} \cong \bbS^{\dim(\sigma)-1} \ast \bbD^{k-\dim(\sigma)} \cong \bbD^k,\]
and also
\[\partial\left(\sigma \ast \Link_{\bbD^k}\left(\sigma\right)\right) = \partial \sigma \ast \Link_{\bbD^k}\left(\sigma\right)\]
and
\[\partial\left(\partial \sigma \ast \bbD^{k-\dim(\sigma)}\right) = \partial \sigma \ast \partial \bbD^{k-\dim(\sigma)} = \partial \sigma \ast \Link_{\bbD^k}(\sigma).\]
In all of these $=$ means equality of simplicial complexes and $\cong$ means homeomorphism.  It follows that we can homotope
$f$ so as to replace
\[f|_{\sigma \ast \Link_{\bbD^k}(\sigma)}\colon \sigma \ast \Link_{\bbD^k}(\sigma) \rightarrow Y\]
with
\[\left(f|_{\partial \sigma}\right) \ast g \colon \partial \sigma \ast \bbD^{k-\dim(\sigma)} \rightarrow f(\sigma) \ast \cL(f(\sigma),\cB) \subset Y.\]
This eliminates $\sigma$, and by (ii) and the definition of $\cL(f(\sigma),\cB)$ it does not introduce
any new simplices of dimension at least $\dim(\sigma)$ mapping to simplices in $\cB$.
Repeating this over and over, we can ensure that there are no simplices $\sigma$ of $\bbD^k$
with $f(\sigma) \in \cB$, so by (i) we have $f(\bbD^k) \subset X$, as desired.
\end{proof}

\subsection{Disc local-injectivity and bad simplices}
We now discuss a more technical version of Proposition \ref{proposition:avoidbadeasy} for
proving the disc local injectivity property.

\begin{proposition}
\label{proposition:avoidbadhard}
Let $Y$ be a simplicial complex with the disc local injectivity property up to dimension $d$, and let $X \subset Y$ be a subcomplex.
Assume that there exists a set $\cB$ of simplices of $Y$ with the following properties:
\begin{itemize}
\item[(i)] A simplex of $Y$ lies in $X$ if and only if none of its faces are in $\cB$. 
In particular, since simplices are faces of themselves no simplices of $\cB$ lie in $X$.
\item[(ii)] If $\sigma,\sigma' \in \cB$ are such that $\sigma \cup \sigma'$ is a simplex
of $Y$, then $\sigma \cup \sigma' \in \cB$.  Note that $\sigma$ and $\sigma'$ might share vertices,
so $\sigma \cup \sigma'$ might not be $\sigma \ast \sigma'$.
\item[(iii)] For all $\sigma \in \cB$, there exists a subcomplex $\hcL(\sigma,\cB)$ of $Y$ with
$\cL(\sigma,\cB) \subset \hcL(\sigma,\cB)$ such that the following holds:
\begin{itemize}
\item[a.] We have $\partial \sigma \ast \hcL(\sigma,\cB) \subset Y$.
\item[b.] All simplices of $\partial \sigma \ast \hcL(\sigma,\cB)$ that are in $\cB$ lie in
$\partial \sigma$.
\item[c.] The complex $\hcL(\sigma,\cB)$ has the disc local injectivity property
up to dimension $(d-\dim(\sigma))$.
\end{itemize}
\end{itemize}
Then $X$ has the disc local injectivity property up to dimension $d$.
\end{proposition}

\begin{remark}
We will use the disc local injectivity property to prove that spaces are highly connected.
However, it will play an essential role.  Indeed, if we change the hypothesis (resp.\ conclusion)
of Proposition \ref{proposition:avoidbadhard} from $Y$ (resp.\ $X$) having the disc local injectivity property
up to dimension $d$ to $Y$ (resp.\ $X$) being $d$-connected, then the proof would not
work.  We will highlight the issue in the proof below with a footnote.
\end{remark}

\begin{proof}[Proof of Proposition \ref{proposition:avoidbadhard}]
For some $k \leq d$, let $f\colon \bbS^k \rightarrow X$ be a locally injective
simplicial map with respect to some combinatorial triangulation of $\bbS^k$.  Since
$Y$ has the disc local injectivity property up to dimension $d$, there exists
a combinatorial triangulation of $\bbD^{k+1}$ agreeing with our given triangulation
of $\bbS^k$ on the boundary and an extension of $f$ to a locally injective
simplicial map $F\colon \bbD^{k+1} \rightarrow Y$.  Assume that the image of $F$ does not lie in $X$.

By (i), there exists a simplex $\sigma$ of $\bbD^{k+1}$ such that $F(\sigma) \in \cB$.  Pick
$\sigma$ such that its dimension is maximal among all simplices with $F(\sigma) \in \cB$.
Just like in the proof of Proposition \ref{proposition:avoidbadeasy}, we can use
(ii) along with the maximality of $\dim(\sigma)$ to deduce that
$F$ takes $\Link_{\bbD^{k+1}}(\sigma)$ to $\cL(F(\sigma),\cB)$.

Since $F$ takes $\partial \bbD^{k+1}$ to $X$, the simplex $\sigma$ does not
lie in $\partial \bbD^{k+1}$.  Thus the complex $\Link_{\bbD^{k+1}}(\sigma)$ is a combinatorial
triangulation of a $(k-\dim(\sigma))$ sphere, and by
(iii).c the complex $\hcL(F(\sigma),\cB)$ has the disc local injectivity property up
to dimension $(d-\dim(F(\sigma)))$.  Since $F$ is locally injective and
\[k-\dim(\sigma) \leq d-\dim(\sigma),\]
we can find a combinatorial triangulation of $\bbD^{k-\dim(\sigma)+1}$ agreeing with that of
$\Link_{\bbD^{k+1}}(\sigma) \cong \bbS^{k-\dim(\sigma)}$ on
$\partial \bbD^{k-\dim(\sigma)+1}$ and a locally injective map 
$G\colon \bbD^{k-\dim(\sigma)+1} \rightarrow \hcL(F(\sigma),\cB)$
extending the restriction of $F$ to $\Link_{\bbD^{k+1}}(\sigma)$.  

We have
\[\sigma \ast \Link_{\bbD^{k+1}}\left(\sigma\right) \cong \bbD^{\dim(\sigma)} \ast \bbS^{k-\dim(\sigma)} \cong \bbD^{k+1}\]
and
\[\partial \sigma \ast \bbD^{k-\dim(\sigma)+1} \cong \bbS^{\dim(\sigma)-1} \ast \bbD^{k-\dim(\sigma)+1} \cong \bbD^{k+1},\]
and also
\[\partial\left(\sigma \ast \Link_{\bbD^{k+1}}\left(\sigma\right)\right) = \partial \sigma \ast \Link_{\bbD^{k+1}}\left(\sigma\right)\]
and
\[\partial\left(\partial \sigma \ast \bbD^{k-\dim(\sigma)+1}\right) = \partial \sigma \ast \partial \bbD^{k-\dim(\sigma)+1} = \partial \sigma \ast \Link_{\bbD^{k+1}}(\sigma).\]
In all of these $=$ means equality of simplicial complexes and $\cong$ means homeomorphism.  By
By (iii).a,\footnote{This is where local injectivity is key.  If $F$ were not locally injective, then
$F$ might not take $\partial \sigma$ to $\partial F(\sigma)$, so the map
$\left(F|_{\partial \sigma}\right) \ast G$ might not be continuous.}
we can modify\footnote{Unlike in previous such arguments, we cannot achieve this modification
by a homotopy since $\hcL(F(\sigma),\cB)$ might not lie in $\Link_{Y}(F(\sigma))$.} $F$ so as to replace
\[F|_{\sigma \ast \Link_{\bbD^{k+1}}(\sigma)}\colon \sigma \ast \Link_{\bbD^{k+1}}(\sigma) \rightarrow Y\]
with
\[\left(F|_{\partial \sigma}\right) \ast G \colon \partial \sigma \ast \bbD^{k-\dim(\sigma)+1} \rightarrow \partial(F(\sigma)) \ast \hcL(F(\sigma),\cB) \subset Y.\]
This eliminates $\sigma$, and by (ii) and (iii).b it does not introduce
any new simplices of dimension at least $\dim(\sigma)$ mapping to simplices in $\cB$.
Repeating this over and over, we can ensure that there are no simplices $\sigma$ of $\bbD^{k+1}$
with $F(\sigma) \in \cB$, so by (i) we have $F(\bbD^{k+1}) \subset X$, as desired.
\end{proof}

\section{The complex of spheres \texorpdfstring{$\Spheres(M_{n,b},H)$}{S(M,H)}}
\label{section:spheres}

We now begin our discussion of the topology of the sphere complex and its subcomplexes.

\subsection{Basic definitions}
We start by recalling the following two definitions from the introduction.

\begin{definition}
Let $M_{n,b}$ be the connect sum of $n$ copies of $S^2 \times S^1$ with
$b$ disjoint open balls removed.  Our convention is that the connect sum of
$0$ copies of $S^2 \times S^1$ is the unit $S^3$ for the connect sum operation,
so $M_{0,b}$ is $S^3$ with $b$ disjoint open balls removed.
\end{definition}

\begin{definition}
A $2$-sphere embedded in $M_{n,b}$ is {\em essential} if it is not homotopic to $\partial M_{n,b}$ or a point.
A {\em rank-$k$ sphere system} in $M_{n,b}$ is a set $\{S_0,\ldots,S_k\}$ of distinct isotopy classes of
essential $2$-spheres embedded in $M_{n,b}$ that can be realized disjointly.
\end{definition}

Throughout this paper, we will identify isotopic essential $2$-spheres and sphere systems.

\subsection{Compatibility with free factors}
We will need to study the relationship between sphere systems and free factors in $\pi_1(M_{n,b})$.
We start with the following two definitions.

\begin{definition}
For some $n \geq 0$ and $b \geq 1$, fix a basepoint $\ast \in \partial M_{n,b}$.  Consider
a sphere system $\sigma = \{S_0,\ldots,S_k\}$ in $M_{n,b}$.  Let $N$ be an open regular neighborhood
of $S_0 \cup \cdots \cup S_k$.  The {\em components of the complement}
of $\sigma$ are the connected components of $M_{n,b} \setminus N$.  The component $X$ of the complement
with $\ast \in \partial X$ is the {\em basepoint-containing component of the complement}.
\end{definition}

Letting the notation be as in the previous definition, if $X$ is the basepoint-containing
component of the complement of $\sigma$, then the map $\pi_1(X,\ast) \rightarrow \pi_1(M_{n,b},\ast)$
is injective and its image is a free factor of $\pi_1(M_{n,b},\ast) \cong F_n$.
We will identify $\pi_1(X,\ast)$ with its image in $\pi_1(M_{n,b},\ast)$.  This
allows the following definition.

\begin{definition}
For some $n \geq 0$ and $b \geq 1$, fix a basepoint $\ast \in \partial M_{n,b}$ 
and let $H < \pi_1(M_{n,b},\ast) \cong F_n$ be a free factor.
A sphere system $\sigma$ in $M_{n,b}$ is {\em $H$-compatible} if the following holds.  Let $X$ be the
basepoint-containing component of the complement of $\sigma$.  We then require that
$H \subset \pi_1(X,\ast)$.
\end{definition}

\subsection{Complex of compatible spheres}
This finally brings us to the following key definition.

\begin{definition}
For some $n \geq 0$ and $b \geq 1$, fix a basepoint $\ast \in \partial M_{n,b}$ 
and let $H < \pi_1(M_{n,b},\ast) \cong F_n$ be a free factor.  The {\em complex of
$H$-compatible spheres} in $M_{n,b}$, denoted $\Spheres(M_{n,b},H)$, is the simplicial
complex whose $k$-simplices are isotopy classes of $H$-compatible rank-$k$ sphere systems
in $M_{n,b}$.  If $H = 1$, then we will sometimes omit it from our notation and just
write $\Spheres(M_{n,b})$.
\end{definition}

This complex was introduced by Hatcher--Vogtmann \cite{HatcherVogtmann, HatcherVogtmannRevised}, who proved
the following theorem:

\begin{theorem}[{Hatcher--Vogtmann, \cite{HatcherVogtmannRevised}}]
\label{theorem:spherescon}
For some $n,b \geq 1$, fix a basepoint $\ast \in \partial M_{n,b}$ 
and let $H < \pi_1(M_{n,b},\ast) \cong F_n$ be a free factor.  Assume
that $\rank(H) \leq n-1$.  Then $\Spheres(M_{n,b},H)$ is contractible.
\end{theorem}

\begin{remark}
For $H = 1$, this was originally proved by Hatcher \cite{HatcherStabilization}.
\end{remark}

\begin{remark}
\label{remark:onlyone}
In the published version \cite{HatcherVogtmann} of their paper, Hatcher--Vogtmann only prove the case $b=1$ of
Theorem \ref{theorem:spherescon}.  In 2022, they posted the revised version \cite{HatcherVogtmannRevised}
to the arXiv.  This version fixes an error (see Remark \ref{remark:hatchervogtmannerror} below), and
contains the general case of Theorem \ref{theorem:spherescon}.  This is done in two steps:
\cite[Theorem 2.1]{HatcherVogtmannRevised} proves the case $b=1$, and \cite[Lemma 2.3]{HatcherVogtmannRevised}
proves that $\Spheres(M_{n,b+1},H)$ deformation retracts to a complex isomorphic to $\Spheres(M_{n,b},H)$
for $b \geq 1$, proving the case $b \geq 2$.
\end{remark}

\begin{remark}
\label{remark:genuszero}
Theorem \ref{theorem:spherescon} fails for $n=0$.  Since $\pi_1(M_{0,b},\ast) = 1$, the free factor $H$ is irrelevant here and we
will omit it.  The complex $\Spheres(M_{0,b})$ is $(b-4)$-dimensional, and it turns out that it
is homotopy equivalent to a wedge of $(b-4)$-dimensional spheres, i.e., is $(b-5)$-connected.
See\footnote{This corresponds to the case $k=0$ and $C = \emptyset$ of part (1) of the proof
of \cite[Theorem 3.1]{HatcherWahlBoundaries}.  This theorem concerns a complex of discs and spheres,
but since $C = \emptyset$ we are not allowing any discs and it reduces to our complex.}
\cite[proof of Theorem 3.1]{HatcherWahlBoundaries}.
\end{remark}

\subsection{Low complexity cases}
When $\dim(H) = n-1$ and $b=1$, the complex $\Spheres(M_{n,b},H)$ has a particularly simple description:

\begin{lemma}
\label{lemma:changeofcoordinates}
For some $n \geq 1$, fix a basepoint $\ast \in \partial M_{n,1}$
and let $H < \pi_1(M_{n,1},\ast) \cong F_n$ be a free factor with $\rank(H) = n-1$.
The following then hold:
\begin{itemize}
\item If $n=1$ (so $H = 1$), then $\Spheres(M_{1,1})$ consists of a single vertex corresponding
to a nonseparating sphere in $M_{1,1}$.
\item If $n \geq 2$, then $\Spheres(M_{n,1},H)$ consists of two vertices joined by an edge, one
vertex corresponding to a nonseparating sphere and the other to a separating sphere.
\end{itemize}
\end{lemma}
\begin{proof}
We start by proving that $\Spheres(M_{1,1})$ consists of a single vertex corresponding to 
a nonseparating sphere in $M_{1,1}$.  First note that by basic $3$-manifold topology (in particular,
the uniqueness of the connect sum decomposition, see \cite{MilnorConnectSum} or \cite[\S 3]{Hempel}), we have the
following:
\begin{itemize}
\item Any separating $2$-sphere in $M_{1,1}$ is parallel to $\partial M_{1,1}$, and is thus
not essential.  It follows that the vertices of $\Spheres(M_{1,1})$ correspond to nonseparating
$2$-spheres.
\item Cutting $M_{1,1}$ open along a nonseparating $2$-sphere yields $M_{0,3}$.  From this, we see that the mapping class group
group of $M_{1,1}$ acts transitively on the set of isotopy classes of nonseparating
$2$-spheres in $M_{1,1}$.  For more details, see the discussion of the ``change of coordinates'' principle from \cite[\S 1.3]{FarbMargalitPrimer}, which concerns
the related case of mapping class groups of surfaces.  
\end{itemize}
Combining these two facts, it is enough to prove that there is a single mapping class group
orbit of nonseparating $2$-sphere in $M_{1,1}$.
As we discussed in the introduction,
Laudenbach \cite{LaudenbachPaper, LaudenbachBook} proved that
the action of the mapping class group of $M_{1,1}$ on the set of isotopy classes of $2$-spheres
factors through $\Aut(F_1)$.
We have $\Aut(F_1) = \Z/2$, generated by the automorphism of $F_1 = \Z$ that takes
$1$ to $-1$.  This clearly fixes the ``core'' sphere of
\[M_{1,1} = S^2 \times S^1 \setminus \text{ball},\]
so this is the unique nonseparating $2$-sphere up to isotopy, as desired.

We now turn to the case $n \geq 2$.  Theorem \ref{theorem:spherescon} says that $\Spheres(M_{n,1},H)$ is connected,
so to prove the lemma it is enough to prove that each vertex $S$ of $\Spheres(M_{n,1},H)$ is contained
in a unique edge.  The vertex $S$ is an essential $H$-compatible $2$-sphere in $M_{n,1}$.  If $S$
is a separating $2$-sphere, then since $\rank(H) = n-1$ we must have that $S$ separates
$M_{n,1}$ into components $X$ and $Y$ with $X \cong M_{1,1}$ and $Y \cong M_{n-1,2}$.  The
component $Y$ is the basepoint-containing component and satisfies $\pi_1(Y) = H$.  An edge
in $\Spheres(M_{n,1},H)$ must connect $S$ to an essential $2$-sphere contained in
$X$, and by the previous paragraph there is a unique such essential $2$-sphere.  It follows that
$S$ is contained in a unique edge, as desired.

If instead $S$ is a nonseparating $2$-sphere, then there is a single component $Z$ of the complement
of $S$.  We have $Z \cong M_{n-1,3}$, and $\pi_1(Z,\ast) = H$.  An edge in $\Spheres(M_{n,1},H)$
must connect $S$ to an $H$-compatible separating $2$-sphere $T$ in $Z$.  Letting $S'$ and $S''$
be the boundary components of $Z$ that are parallel to $S$ in $M_{n,1}$, such a $T$ must be
the boundary of a regular neighborhood of $S' \cup S'' \cup \alpha$, where $\alpha$ is an arc
connecting $S'$ to $S''$.  By the lightbulb trick (see, e.g., \cite[Exercise 9.F.4]{RolfsenKnots}), there is a unique such arc up to isotopy, so
there is unique such $T$, as desired.
\end{proof}

\subsection{Dual graph to vertex}
The following notion will be very useful.

\begin{definition}
Let $\sigma$ be a sphere system in $M_{n,b}$.  The {\em dual graph} of $\sigma$, denoted $\Gamma(\sigma)$,
is following graph:
\begin{itemize}
\item The vertices of $\Gamma(\sigma)$ are the components of the complement of $\sigma$.
\item The edges of $\Gamma(\sigma)$ are in bijection with the spheres in $\sigma$, and the
edge corresponding to $S \in \sigma$ connects the vertices corresponding to the components
on either side of $S$.\qedhere
\end{itemize}
\end{definition}

These satisfy the following lemma:

\begin{lemma}
\label{lemma:dualgraph}
Let $\sigma$ be a sphere system in $M_{n,b}$.  Let $X_1,\ldots,X_r$ be the components
of the complement of $\sigma$.  Write $X_i \cong M_{n_i,b_i}$.  Then
\[n = \rank(\pi_1(\Gamma(\sigma))) + \sum_{i=1}^r n_i.\]
\end{lemma}
\begin{proof}
Immediate.
\end{proof}

\subsection{Cohen--Macaulay}
Using Theorem \ref{theorem:spherescon}, we can prove the following.

\begin{theorem}
\label{theorem:spherescm}
For some $n,b \geq 1$, fix a basepoint $\ast \in \partial M_{n,b}$
and let $H < \pi_1(M_{n,b},\ast) \cong F_n$ be a free factor.  Assume that $(n,b) \neq (1,1)$ and
that $\rank(H) \leq n-1$.  Then $\Spheres(M_{n,b},H)$ is weakly Cohen--Macaulay of dimension
$n-\rank(H)$.
\end{theorem}

\begin{remark}
\label{remark:spherescmboundary}
The conditions that $(n,b) \neq (1,1)$ and $\rank(H) \leq n-1$ are necessary.  Indeed, if $\rank(H) = n$ then
$\Spheres(M_{n,1},H)$ is the empty set, and thus is not weakly Cohen--Macaulay of dimension $0$.  Similarly,
by Lemma \ref{lemma:changeofcoordinates} the
complex $\Spheres(M_{1,1})$ is a single point.  Thus while it is contractible (and hence connected), it
is not weakly Cohen--Macaulay of dimension $1$.
\end{remark} 

\begin{proof}[Proof of Theorem \ref{theorem:spherescm}]
Given our assumptions, Theorem \ref{theorem:spherescon} implies that $\Spheres(M_{n,b},H)$
is contractible, and thus is certainly $(n-1-\rank(H))$-connected.  Letting $\sigma = \{S_0,\ldots,S_k\}$
be a $k$-simplex of $\Spheres(M_{n,b},H)$, we must prove that $\Link_{\Spheres(M_{n,b},H)}(\sigma)$
is $(n-2-k-\rank(H))$-connected.  What we will prove is that either
\begin{itemize}
\item $\Link_{\Spheres(M_{n,b},H)}(\sigma)$ is contractible, or
\item $n = k + 1 + \rank(H)$ and $\Link_{\Spheres(M_{n,b},H)}(\sigma)$ is nonempty, or
\item $n \leq k + \rank(H)$.
\end{itemize}
This will imply that $\Link_{\Spheres(M_{n,b},H)}(\sigma)$ is always $(n-2-k-\rank(H))$-connected.

Let the components of the complement of $\sigma$ be
$X_1,\ldots,X_r$ with $X_1$ the basepoint-containing component.  We then have that
\[\Link_{\Spheres(M_{n,b},H)}(\sigma) = \Spheres(X_1,H) \ast \Spheres(X_2) \ast \cdots \ast \Spheres(X_r).\]
Write $X_i = M_{n_i,b_i}$.  If $\rank(H) \leq n_1 - 1$, then Theorem \ref{theorem:spherescon} implies
that $\Spheres(X_1,H)$ is contractible, so $\Link_{\Spheres(M_{n,b},H)}(\sigma)$ is contractible.
Similarly, if $n_i \geq 1$ for
any $2 \leq i \leq r$, then Theorem \ref{theorem:spherescon} implies that
$\Spheres(X_i)$ is contractible, so $\Link_{\Spheres(M_{n,b},H)}(\sigma)$ is contractible.
We can therefore assume without loss of generality that $n_1 = \rank(H)$ and that $n_i = 0$ for
$2 \leq i \leq r$.  Our goal then reduces to proving that
$n \leq k + 1 + \rank(H)$, and that if $n = k + 1 + \rank(H)$ then $\Link_{\Spheres(M_{n,b},H)}(\sigma)$
is nonempty.

Letting $\Gamma(\sigma)$ be the dual graph of $\sigma$, Lemma \ref{lemma:dualgraph} implies that
\[n = \rank(\pi_1(\Gamma(\sigma))) + \sum_{i=1}^r n_i = \rank(\pi_1(\Gamma(\sigma))) + \rank(H).\]
Since $\sigma$ contains $(k+1)$ spheres, the graph $\Gamma(\sigma)$ has $(k+1)$ edges.  It follows
that $\rank(\pi_1(\Gamma(\sigma))) \leq k+1$, so
\[n \leq k+1 + \rank(H).\]
This is half of what we are trying to prove.  What remains is to show that if this inequality
is an equality, then $\Link_{\Spheres(M_{n,b},H)}(\sigma)$
is nonempty.

For our inequality to be an equality, we must have $\rank(\pi_1(\Gamma(\sigma))) = k+1$.  Since
$\Gamma(\sigma)$ has $(k+1)$ edges, this can only happen if $\Gamma(\sigma)$ has a single
vertex and all the edges are self-loops.  In other words, there is only one component
of the complement of $\sigma$, namely the basepoint-containing component $X_1$.  We thus
have
\[X_1 \cong M_{n-k-1,b+2k+2} \quad \text{and} \quad H = \pi_1(X_1).\]
If either $b \geq 2$ or $k \geq 1$, then $X_1$ has at least $4$ boundary components, so $X_1$ has an $H$-compatible
separating\footnote{For later use, note that if $k \geq 1$ then we can choose the $H$-compatible separating sphere in $X_1$
such that it becomes nonseparating in the larger $3$-manifold $M_{n,b}$.} sphere that cuts off two boundary spheres and thus
\[\Link_{\Spheres(M_{n,b},H)}(\sigma) \cong \Spheres(X_1,H) \neq \emptyset.\]
If instead $b = 1$ and $k = 0$, then $X_1 \cong M_{n-1,3}$.  This
is where we finally invoke our assumption that $(n,b) \neq (1,1)$, which implies that
$X_1$ is not just a $3$-holed sphere, so $X_1$ has an $H$-compatible
separating sphere that cuts off two boundary spheres and
\[\Link_{\Spheres(M_{n,b},H)}(\sigma) \cong \Spheres(X_1,H) \neq \emptyset.\qedhere\]
\end{proof}

\section{The nonseparating complex of spheres \texorpdfstring{$\NSpheres(M_{n,b},H)$}{N(M,H)}}
\label{section:nsspheres}

The next step is to consider several subcomplexes of $\Spheres(M_{n,b},H)$.

\subsection{Nonseparating simplices}
We start with the following.

\begin{definition}
For some $n \geq 0$ and $b \geq 1$, fix a basepoint $\ast \in \partial M_{n,b}$
and let $H < \pi_1(M_{n,b},\ast) \cong F_n$ be a free factor.  The {\em nonseparating complex of
$H$-compatible spheres} in $M_{n,b}$, denoted $\NSpheres(M_{n,b},H)$, is
the subcomplex of $\Spheres(M_{n,b},H)$ whose $k$-simplices are $H$-compatible rank-$k$ sphere
systems $\{S_0,\ldots,S_k\}$ in $M_{n,b}$ such that $S_0 \cup \cdots \cup S_k$ does not
separate $M_{n,b}$.
\end{definition}

\subsection{High connectivity}
The following theorem of Hatcher--Vogtmann \cite{HatcherVogtmann, HatcherVogtmannRevised} says that these complexes are highly connected:

\begin{theorem}[{Hatcher--Vogtmann \cite{HatcherVogtmannRevised}}]
\label{theorem:nspherescon}
For some $n,b \geq 1$, fix a basepoint $\ast \in \partial M_{n,b}$
and let $H < \pi_1(M_{n,b},\ast) \cong F_n$ be a free factor.  Then,
$\NSpheres(M_{n,b},H)$ is $(n-\rank(H)-2)$-connected.
\end{theorem}

\begin{remark}
\label{remark:hatchervogtmannerror}
Hatcher--Vogtmann's published proof of Theorem \ref{theorem:nspherescon} in \cite{HatcherVogtmann}
has a mistake: \cite[Lemma 2.3]{HatcherVogtmann} claims that $\NSpheres(M_{n,b+1},H)$
deformation retracts to $\NSpheres(M_{n,b},H)$, which is false (e.g., $\NSpheres(M_{1,1})$ is a single
point, but $\NSpheres(M_{1,2})$ is an infinite discrete set).  The issue is that their
deformation retraction uses simplices that do not lie in $\NSpheres(M_{n,b+1},H)$.  
In 2022, they posted a revised version \cite{HatcherVogtmannRevised} of their paper to the arXiv
correcting this mistake.  Theorem \ref{theorem:nspherescon} is \cite[Theorem 2.5]{HatcherVogtmannRevised}.
\end{remark}

\subsection{Cohen--Macaulay}
Using Theorem \ref{theorem:nspherescon}, we can prove the following.

\begin{theorem}
\label{theorem:nspherescm}
For some $n,b \geq 1$, fix a basepoint $\ast \in \partial M_{n,b}$
and let $H < \pi_1(M_{n,b},\ast) \cong F_n$ be a free factor.
Then $\NSpheres(M_{n,b},H)$ is weakly Cohen--Macaulay of dimension $n-\rank(H)-1$.
\end{theorem}
\begin{proof}
Immediate from Theorem \ref{theorem:nspherescon} along with the fact that if
$\sigma = \{S_0,\ldots,S_k\}$ is a $k$-simplex of $\NSpheres(M_{n,b},H)$, then
\[\Link_{\NSpheres(M_{n,b},H)}(\sigma) \cong \NSpheres(M_{n-k-1,b+2k+2},H).\]
The point here is that the only connected component of the complement of $\sigma$
is homeomorphic to $M_{n-k-1,b+2k+2}$.
\end{proof}

\section{The augmented nonseparating complex of spheres \texorpdfstring{$\ANSpheres(M_{n,b},H)$}{AN(M,H)} I: sphere local injectivity}
\label{section:augmentedsphere}

We now add some simplices to the nonseparating complex of spheres $\NSpheres(M_{n,b},H)$ to increase its connectivity
by one.

\subsection{Separating core}
This requires the following definition.

\begin{definition}
Let $\sigma$ be a sphere system on $M_{n,b}$ with dual graph $\Gamma(\sigma)$.  The
{\em separating core} of $\sigma$ is the face $\sigma'$ of $\sigma$ consisting of all
$S \in \sigma$ such that the edge of $\Gamma(\sigma)$ corresponding to $S$ is not a loop.
The {\em nonseparating periphery} of $\sigma$ is the face of $\sigma$ consisting of all
$S \in \sigma$ such that the edge of $\Gamma(\sigma)$ corresponding to $S$ is a loop. 
\end{definition}

Another way to think about this is as follows.  Let $\sigma$ be a sphere system on $M_{n,b}$
with separating core $\sigma'$ and nonseparating periphery $\sigma''$.  As simplices
of $\Spheres(M_{n,b})$, we have $\sigma = \sigma' \ast \sigma''$.  The following hold:
\begin{itemize}
\item Any proper face of $\sigma'$ has strictly fewer components in its complement than
$\sigma'$.
\item Let $X_1,\ldots,X_r$ be the components of the complement of $\sigma'$.  The link
of $\sigma'$ in $\Spheres(M_{n,b})$ is thus
\[\Spheres(X_1) \ast \cdots \ast \Spheres(X_r).\]
Then $\sigma''$ lies in the subcomplex
\[\NSpheres(X_1) \ast \cdots \ast \NSpheres(X_r) \subset \Spheres(X_1) \ast \cdots \ast \Spheres(X_r).\]
\end{itemize}

\subsection{Augmented complex}
The definition of our complex is as follows.

\begin{definition}
For some $n \geq 0$ and $b \geq 1$, fix a basepoint $\ast \in \partial M_{n,b}$
and let $H < \pi_1(M_{n,b},\ast) \cong F_n$ be a free factor.  The {\em augmented nonseparating complex of
$H$-compatible spheres} in $M_{n,b}$, denoted $\ANSpheres(M_{n,b},H)$, is
the subcomplex of $\Spheres(M_{n,b},H)$ whose $k$-simplices are $H$-compatible rank-$k$ sphere
systems $\sigma = \{S_0,\ldots,S_k\}$ in $M_{n,b}$ with the following properties:
\begin{itemize}
\item Each $S_i$ is a nonseparating sphere.
\item Let $\sigma'$ be the separating core of $\sigma$.  Then either $\sigma'$ is empty (so
$\sigma$ does not separate $M_{n,b}$), or $\sigma'$ has two components $X$ and $Y$ in
its complement.  Moreover, in the latter case if $X$ is the basepoint-containing component
then $\pi_1(Y) = 1$.\qedhere
\end{itemize}
\end{definition}

\subsection{Connectivity theorem}
Recall that Theorem \ref{theorem:nspherescon} says that $\NSpheres(M_{n,b},H)$ is $(n-2-\rank(H))$-connected.
Our main theorem about the augmented complexes is as follows:

\begin{theorem}
\label{theorem:nspheresredcon}
For some $n,b \geq 1$, fix a basepoint $\ast \in \partial M_{n,b}$
and let $H < \pi_1(M_{n,b},\ast) \cong F_n$ be a free factor.  Assume
that $\rank(H) \leq n-1$.  Then $\ANSpheres(M_{n,b},H)$
is $(n-1-\rank(H))$-connected.
\end{theorem}

\subsection{A stronger result}
Some degenerate cases of Theorem \ref{theorem:nspheresredcon} are immediate:
\begin{itemize}
\item If $b=1$ and $\rank(H) = n-1$, then by Lemma \ref{lemma:changeofcoordinates}
the complex $\ANSpheres(M_{n,b},H)$ consists of a single vertex
corresponding to a nonseparating curve.  It is thus contractible, and in particular is
$(n-1-\rank(H))=0$ connected.
\end{itemize}
We thus can exclude these cases.  In the remaining ones, we will actually prove the
following stronger result.

\begin{theorem}
\label{theorem:anspheresbetter}
For some $n,b \geq 1$, fix a basepoint $\ast \in \partial M_{n,b}$
and let $H < \pi_1(M_{n,b},\ast) \cong F_n$ be a free factor.  Assume
that $\rank(H) \leq n-1$ and that if $\rank(H) = n-1$, then $b \geq 2$.
Then $\ANSpheres(M_{n,b},H)$ has
the sphere and disc local injectivity properties up to dimension $(n-1-\rank(H))$.
\end{theorem}

By Lemma \ref{lemma:localinjcon}, this will imply that $\ANSpheres(M_{n,b},H)$ is $(n-1-\rank(H))$-connected.

\begin{remark}
Theorem \ref{theorem:anspheresbetter} is weaker than saying that these complexes
are weakly Cohen--Macaulay.  We do not know if this strong condition holds.
\end{remark}

We divide the proof of Theorem \ref{theorem:anspheresbetter} into two parts: in the rest
of this section, we prove sphere local injectivity (see Lemma \ref{lemma:augsphere}), and
in \S \ref{section:augmenteddiscunred} -- \S \ref{section:augmenteddisc} we prove disc local injectivity
(see Lemma \ref{lemma:augdisc}).

\subsection{Sphere local injectivity}
The following result takes care of the sphere local injectivity part of
Theorem \ref{theorem:anspheresbetter}.  We note that the hypothesis
that $b \geq 2$ if $\rank(H) = n-1$ is not needed here, but will be
needed for the disc local injectivity property.

\begin{lemma}
\label{lemma:augsphere}
For some $n,b \geq 1$, fix a basepoint $\ast \in \partial M_{n,b}$
and let $H < \pi_1(M_{n,b},\ast) \cong F_n$ be a free factor.  Assume
that $\rank(H) \leq n-1$.
Then $\ANSpheres(M_{n,b},H)$ has
the sphere local injectivity property up to dimension $(n-1-\rank(H))$.
\end{lemma}
\begin{proof}
Theorem \ref{theorem:nspherescm} says that $\NSpheres(M_{n,b},H)$ is weakly Cohen--Macaulay of dimension
$(n-1-\rank(H))$, so by Lemma \ref{lemma:makeinjectivecm} it has the sphere local injectivity
property up to dimension $(n-1-\rank(H))$.  It follows that it is enough to prove that
for $k \leq n-1-\rank(H)$, every map $\bbS^k \rightarrow \ANSpheres(M_{n,b},H)$ is
homotopic to one whose image lies in $\NSpheres(M_{n,b},H)$.  Letting $d = n-1-\rank(H)$, this
is equivalent to proving that the inclusion
\[\NSpheres(M_{n,b},H) \hookrightarrow \ANSpheres(M_{n,b},H)\]
is $d$-connected.
This will follow from Proposition \ref{proposition:avoidbadeasy}
once we verify its hypotheses.
The input to Proposition \ref{proposition:avoidbadeasy} is a set $\cB$ of ``bad simplices'', which
for us will be as follows:
\begin{itemize}
\item The set $\cB$ consist of all simplices $\sigma = \{S_0,\ldots,S_k\}$ of $\ANSpheres(M_{n,b},H)$
such that the union of the $S_i$ separates $M_{n,b}$, but no proper subset
of $\sigma$ separates $M_{n,b}$.  In other words, $\sigma$ is its own separating core.
\end{itemize}
We now verify each hypothesis of Proposition \ref{proposition:avoidbadeasy} in turn.

Condition (i) says that a simplex of $\ANSpheres(M_{n,b},H)$ lies in $\NSpheres(M_{n,b},H)$ if
and only if none of its faces are in $\cB$, which is immediate from the definitions.

Condition (ii) says that if $\sigma,\sigma' \in \cB$ are such that $\sigma \cup \sigma'$ is a simplex
of $\ANSpheres(M_{n,b},H)$, then $\sigma \cup \sigma' \in \cB$.  In fact, since a simplex of
$\ANSpheres(M_{n,b},H)$ can separate $M_{n,b}$ into at most $2$ components, the only way
that $\sigma \cup \sigma'$ can be a simplex of $\ANSpheres(M_{n,b},H)$ is for $\sigma = \sigma'$,
so this condition is trivial.

Condition (iii) says that for all $\sigma \in \cB$, the complex
$\cL(\sigma,\cB)$ defined in Definition \ref{definition:badlink} is
\[d-\dim(\sigma)-1 = (n-1-\rank(H)) - \dim(\sigma) - 1 = n-2-\rank(H) - \dim(\sigma)\]
connected.  For our $\cB$, the complex $\cL(\sigma,\cB)$
has the following concrete description.  Let $X$ and $Y$ be the components
of the complement of $\sigma$, with $X$ the basepoint-containing component.  Since
$\pi_1(Y) = 1$, all $2$-spheres in $Y$ separate $Y$.  We thus have
\begin{equation}
\label{eqn:anspheresjoin}
\cL(\sigma,\cB) = \NSpheres(Y) \ast \NSpheres(X,H) = \NSpheres(X,H).
\end{equation}
Write $X \cong M_{n',b'}$.  Theorem \ref{theorem:nspherescon} says that $\cL(\sigma,\cB) = \NSpheres(X,H)$
is $(n'-2-\rank(H))$-connected, so to prove that it is $(n-2-\rank(H) - \dim(\sigma))$-connected
we must prove that $n = n' + \dim(\sigma)$.

The dual graph $\Gamma(\sigma)$ has two vertices corresponding to $X$ and $Y$ and $\dim(\sigma)+1$
edges, so
\[\pi_1(\Gamma(\sigma)) \cong F_{\dim(\sigma)}.\]
Lemma \ref{lemma:dualgraph} now says that
\[n = \rank(\pi_1(\Gamma)) + \rank(\pi_1(X)) + \rank(\pi_1(Y)) = \dim(\sigma) + n' + 0,\]
as desired.
\end{proof}

\section{The complex of nonseparating spheres \texorpdfstring{$\SpheresNS(M_{n,b},H)$}{Sns(M,H)}}
\label{section:spheresns}

Before we can prove disc local injectivity for the augmented nonseparating sphere complex,
we must study the subcomplex of the sphere complex where vertices are nonseparating, but
where higher-dimensional simplices can separate.

\subsection{Nonseparating spheres, absolute version}
The definition is as follows.

\begin{definition}
\label{definition:nonseparatingspheres}
For some $n \geq 0$ and $b \geq 1$, fix a basepoint $\ast \in \partial M_{n,b}$
and let $H < \pi_1(M_{n,b},\ast) \cong F_n$ be a free factor.  The {\em complex of
$H$-compatible nonseparating spheres} in $M_{n,b}$, denoted $\SpheresNS(M_{n,b},H)$, is 
the full subcomplex of $\Spheres(M_{n,b},H)$ whose vertices are the isotopy classes
of $H$-compatible essential nonseparating spheres in $M_{n,b}$.
If $H = 1$, then we will sometimes omit it from our notation and just
write $\SpheresNS(M_{n,b})$.
\end{definition}

\subsection{Relative version}
Our main theorem about $\SpheresNS(M_{n,b},H)$ is that it is weakly Cohen--Macaulay of dimension
$n-\rank(H)$, just like $\Spheres(M_{n,b},H)$ (c.f.\ Theorem \ref{theorem:spherescm}).  A technical
issue that will arise when studying links in $\SpheresNS(M_{n,b},H)$ is that a sphere can separate
a submanifold of a $3$-manifold without separating the whole manifold.  We thus make
the following definition:

\begin{definition}
For some $n \geq 0$ and $b \geq 1$, fix a basepoint $\ast \in \partial M_{n,b}$
and let $H < \pi_1(M_{n,b},\ast) \cong F_n$ be a free factor.  Let $M$ be
another connected 3-manifold with boundary such that $M_{n,b}$ is a submanifold of $M$.
The {\em complex of $H$-compatible $M$-nonseparating spheres} in $M_{n,b}$, denoted $\SpheresNS(M_{n,b},M,H)$, is the 
full subcomplex of $\Spheres(M_{n,b},H)$ whose vertices are the isotopy classes
of $H$-compatible essential spheres in $M_{n,b}$ that do not separate $M$.
If $H = 1$, then we will sometimes omit it from our notation and just
write $\SpheresNS(M_{n,b},M)$.
\end{definition}

\subsection{Contractibility.}
The following is the analogue for $\SpheresNS(M_{n,b},M,H)$ of Theorem \ref{theorem:spherescon}
for $\Spheres(M_{n,b},H)$:

\begin{theorem}
\label{theorem:spheresnscon}
For some $n,b \geq 1$, fix a basepoint $\ast \in \partial M_{n,b}$
and let $H < \pi_1(M_{n,b},\ast) \cong F_n$ be a free factor.  Assume
that $\rank(H) \leq n-1$, and let $M$ be another connected $3$-manifold
with boundary such that $M_{n,b}$ is a submanifold of $M$.  
Then $\SpheresNS(M_{n,b},M,H)$ is contractible.
\end{theorem}
\begin{proof}
The proof will be by induction on the pair $(n,b)$, ordered lexicographically.  For the base case $(n,b) = (1,1)$, our
assumption on $H$ means that $H = 1$.  By Lemma \ref{lemma:changeofcoordinates}, the
complex
\[\SpheresNS(M_{1,1},M) = \Spheres(M_{1,1})\]
is a single vertex represented by a nonseparating sphere, and is thus contractible.  

Assume, therefore, that $(n,b) \neq (1,1)$ and that the theorem is true whenever $(n,b)$ is smaller.
Theorem \ref{theorem:spherescon} says that $\Spheres(M_{n,b},H)$ is contractible,
so it is enough to prove that the inclusion 
\[\SpheresNS(M_{n,b},M,H) \subset \Spheres(M_{n,b},H)\]
is $d$-connected for all $d \geq 0$.  This will follow from Proposition \ref{proposition:avoidbadeasy}
once we verify its hypotheses.
The input to Proposition \ref{proposition:avoidbadeasy} is a set $\cB$ of ``bad simplices'', which
for us will be as follows:
\begin{itemize}
\item Let $\cB$ be the set of all simplices $\sigma$ of $\Spheres(M_{n,b},H)$ such that
each vertex of $\sigma$ separates the ambient manifold $M$ (and thus is not a vertex of
$\SpheresNS(M_{n,b},M,H)$).
\end{itemize}
We now verify each hypothesis of Proposition \ref{proposition:avoidbadeasy} in turn.

Condition (i) says that a simplex of $\Spheres(M_{n,b},H)$ lies in $\SpheresNS(M_{n,b},M,H)$ if
and only if none of its faces are in $\cB$, which is immediate from the definitions.

Condition (ii) says that if $\sigma,\sigma' \in \cB$ are such that $\sigma \cup \sigma'$ is a simplex
of $\Spheres(M_{n,b},H)$, then $\sigma \cup \sigma' \in \cB$.  Again, this is immediate from
the definitions.

Condition (iii) says that for all $\sigma \in \cB$, the complex
$\cL(\sigma,\cB)$ defined in Definition \ref{definition:badlink} is $(d-\dim(\sigma)-1)$-connected
for all $d \geq 0$, i.e., is contractible.  For our $\cB$, the complex $\cL(\sigma,\cB)$
has the following concrete description.  Let $X_1,\ldots,X_r$ be the components of the
complement of $\sigma$ in $M_{n,b}$, with $X_1$ the basepoint-containing component.  We
then have
\begin{equation}
\label{eqn:spheresnsjoin}
\cL(\sigma,\cB) = \SpheresNS(X_1,M,H) \ast \SpheresNS(X_2,M) \ast \cdots \ast \SpheresNS(X_r,M).
\end{equation}
To prove that $\cL(\sigma,\cB)$ is contractible, it is enough to prove that at least one
term in this join is contractible.

Write $X_i \cong M_{n_i,b_i}$.  Recall that we are inducting on the pair $(n,b)$, ordered
lexicographically.  Since $X_i$ is a submanifold of $M_{n,b}$, we must have $n_i \leq n$.  Moreover,
if $n_i = n$ then since at least one component of $\partial X_i$ is an essential 
separating\footnote{In fact,
it has to separate the ambient manifold $M$, which is even stronger.}
$2$-sphere in $M_{n,b}$ and none of the components of $\partial X_i$ bound
balls in $M_{n,b}$, it follows that we must have $b_i < b$.  In other words, $(n_i,b_i)$ is
strictly less than $(n,b)$ in the lexicographic ordering.

This might lead the reader to think that our inductive hypothesis applies to each
$X_i$, and thus that all the terms in \eqref{eqn:spheresnsjoin} are contractible.  However,
there is an issue: this only works if the corresponding complex satisfies
the hypotheses of our theorem, and this might not hold.  Since we only need one
term in \eqref{eqn:spheresnsjoin} to be contractible, it is enough to prove
that one of the following two things hold:
\begin{itemize}
\item[$(\dagger)$] For some $2 \leq i \leq r$, we have $n_i \geq 1$.  Then the term
\[\SpheresNS(X_i,M) \cong \SpheresNS(M_{n_i,b_i},M)\] 
of \eqref{eqn:spheresnsjoin} satisfies the the hypothesis
of our theorem, so by our inductive hypothesis is contractible.
\item[$(\dagger\dagger)$] We have $n_1 \geq 1$ and $\rank(H) \leq n_1-1$.  Then the term
\[\SpheresNS(X_1,M,H) \cong \SpheresNS(M_{n_1,b_1},M,H)\] 
of \eqref{eqn:spheresnsjoin} satisfies the the hypothesis
of our theorem, so by our inductive hypothesis is contractible.
\end{itemize}
Assume that $(\dagger)$ does not hold, so $n_i = 0$ for $2 \leq i \leq r$.  We will prove that $(\dagger\dagger)$ then holds.
Since each sphere in $\sigma$ separates the manifold $M$, it also separates $M_{n,b}$.  The
dual graph $\Gamma(\sigma)$ is thus a tree.  Lemma \ref{lemma:dualgraph} therefore says that
\[n = \rank(\pi_1(\Gamma(\sigma))) + \sum_{i=1}^r n_i = 0 + n_1 + \sum_{i=2}^r 0 = n_1.\]
We thus have $n_1 = n$, so $n_1 \geq 1$ and $\rank(H) \leq n_1-1$ by assumption and
$(\dagger\dagger)$ holds, as desired.
\end{proof}

\subsection{Cohen--Macaulay}
Using Theorem \ref{theorem:spheresnscon}, we can prove the following.

\begin{theorem}
\label{theorem:spheresnscm}
For some $n,b \geq 1$, fix a basepoint $\ast \in \partial M_{n,b}$
and let $H < \pi_1(M_{n,b},\ast) \cong F_n$ be a free factor.  Assume 
that $\rank(H) \leq n-1$, and if $\rank(H) = n-1$ then assume that $b \geq 2$.
Then $\SpheresNS(M_{n,b},H)$ is weakly Cohen--Macaulay of dimension $n-\rank(H)$.
\end{theorem}

\begin{remark}
The conditions that $\rank(H) \leq n-1$ and that if $\rank(H) = n-1$ then $b \geq 2$ are necessary.  Indeed, if $\rank(H) = n$ then
$\SpheresNS(M_{n,1},H)$ is the empty set, and thus is not weakly Cohen--Macaulay of dimension $0$.  Similarly,
if $\rank(H) = n-1$ then Lemma \ref{lemma:changeofcoordinates} implies that the
complex $\SpheresNS(M_{n,1})$ is a single point.  Thus while it is contractible (and hence connected), it
is not weakly Cohen--Macaulay of dimension $1$.
\end{remark}

\begin{proof}[Proof of Theorem \ref{theorem:spheresnscm}]
The proof is almost identical to that of Theorem \ref{theorem:spherescm}, so we just
indicate the necessary changes:
\begin{itemize}
\item Use Theorem \ref{theorem:spheresnscon} in place of Theorem \ref{theorem:spherescon}.
\item The ``relative'' version of our complex arises as follows.  Consider a simplex $\sigma$
of $\SpheresNS(M_{n,b},H)$.  We wish to understand the link of $\sigma$.  Let
the components of the complement of $\sigma$ be $X_1,\ldots,X_r$ with $X_1$ the basepoint-containing component.
We then have
\[\Link_{\SpheresNS(M_{n,b},H)}(\sigma) = \SpheresNS(X_1,M_{n,b},H) \ast \Spheres(X_2,M_{n,b}) \ast \cdots \ast \Spheres(X_r,M_{n,b})\]
since vertices of the link must not separate $M_{n,b}$ (though they can separate the $X_i$).
\item Finally, the last paragraph of the proof must be adjusted to ensure that the sphere
that arises does not separate $M_{n,b}$.  The key point here is that the final sentence (where
$b=1$ and $k=0$) is not needed due to our assumption that $b \geq 2$ if $\rank(H) = n-1$.\qedhere
\end{itemize}
\end{proof}

\section{The augmented nonseparating complex of spheres \texorpdfstring{$\ANSpheres(M_{n,b},H)$}{AN(M,H)} II: expanded disc local injectivity}
\label{section:augmenteddiscunred}

We now begin studying disc local injectivity for the augmented nonseparating complex of spheres.  The
first step is to establish this for a slightly larger complex.

\subsection{Basic definition}
The definition of our complex is as follows.

\begin{definition}
For some $n \geq 0$ and $b \geq 1$, fix a basepoint $\ast \in \partial M_{n,b}$
and let $H < \pi_1(M_{n,b},\ast) \cong F_n$ be a free factor.  The {\em expanded augmented nonseparating complex of
$H$-compatible spheres} in $M_{n,b}$, denoted $\ANSpheresEX(M_{n,b},H)$, is
the subcomplex of $\Spheres(M_{n,b},H)$ whose $k$-simplices are $H$-compatible rank-$k$ sphere
systems $\{S_0,\ldots,S_k\}$ in $M_{n,b}$ with the following properties:
\begin{itemize}
\item Each $S_i$ is a nonseparating sphere.
\item The union $S_0 \cup \cdots \cup S_k$ either does not separate $M_{n,b}$, or separates
it into exactly two components. \qedhere 
\end{itemize}
\end{definition}

We thus have $\ANSpheres(M_{n,b},H) \subset \ANSpheresEX(M_{n,b},H)$.

\subsection{Disc local injectivity property}
Lemma \ref{lemma:augsphere} says that $\ANSpheres(M_{n,b},H)$ has the sphere
local injectivity property up to dimension $(n-1-\rank(H))$ as long as $\rank(H) \leq n-1$.
The same is true for $\ANSpheresEX(M_{n,b},H)$, with a very similar proof.
We will not need this, but we will need the disc local injectivity property:

\begin{lemma}
\label{lemma:anspheresurdisc}
For some $n,b \geq 1$, fix a basepoint $\ast \in \partial M_{n,b}$
and let $H < \pi_1(M_{n,b},\ast) \cong F_n$ be a free factor.  Assume
that $\rank(H) \leq n-1$ and that if $\rank(H) = n-1$, then $b \geq 2$.
Then $\ANSpheresEX(M_{n,b},H)$ has
the disc local injectivity property up to dimension $(n-1-\rank(H))$.
\end{lemma}
\begin{proof}
Theorem \ref{theorem:spheresnscm} says that $\SpheresNS(M_{n,b},H)$ is weakly Cohen--Macaulay of dimension
$(n-\rank(H))$, so by Lemma \ref{lemma:makeinjectivecm} it has the disc local injectivity
property up to dimension $d=(n-1-\rank(H))$.  To prove the same for
$\ANSpheresEX(M_{n,b},H)$, it is enough to verify the conditions of
Proposition \ref{proposition:avoidbadhard} for the inclusion
\[\ANSpheresEX(M_{n,b},H) \hookrightarrow \SpheresNS(M_{n,b},H).\]
The input to Proposition \ref{proposition:avoidbadhard} is a set $\cB$ of ``bad simplices'', which
for us will be as follows:
\begin{itemize}
\item The set $\cB$ consist of all simplices $\sigma = \{S_0,\ldots,S_k\}$ of $\SpheresNS(M_{n,b},H)$
such that the union of the $S_i$ separates $M_{n,b}$ into $c \geq 3$ components, and all proper
subsets of $\sigma$ separate $M_{n,b}$ into fewer than $c$ components.  In particular, $\sigma$ is
its own separating core.
\end{itemize}
Another way of describing this condition is that the dual graph $\Gamma(\sigma)$ has at least $3$ vertices
and none of its edges are loops.  We now verify each hypothesis of Proposition \ref{proposition:avoidbadhard} in turn.

Condition (i) says that a simplex of $\SpheresNS(M_{n,b},H)$ lies in $\ANSpheresEX(M_{n,b},H)$ if
and only if none of its faces are in $\cB$, which is immediate from the definitions.

Condition (ii) says that if $\sigma,\sigma' \in \cB$ are such that $\sigma \cup \sigma'$ is a simplex
of $\SpheresNS(M_{n,b},H)$, then $\sigma \cup \sigma' \in \cB$.  This is also immediate from the
definitions.

Condition (iii) says that for all $\sigma \in \cB$, there exists a subcomplex $\hcL(\sigma,\cB)$ of
$\SpheresNS(M_{n,b},H)$ with $\cL(\sigma,\cB) \subset \hcL(\sigma,\cB)$ that satisfies conditions
(iii).a--(iii).c.  Here $\cL(\sigma,\cB)$ is the complex defined in Definition \ref{definition:badlink}.
Fixing some $\sigma \in \cB$, we will take $\hcL(\sigma,\cB) = \cL(\sigma,\cB)$.
The complex $\cL(\sigma,\cB)$ has the following concrete description.  Let $X_1,\ldots,X_r$
be the components of the complement of $\sigma$, with $X_1$ the basepoint-containing component.
Then
\[\hcL(\sigma,\cB) = \cL(\sigma,\cB) = \NSpheres(X_1,H) \ast \NSpheres(X_2) \ast \cdots \ast \NSpheres(X_r).\]
We verify (iii).a--(iii).c for this as follows.

Condition (iii).a says that 
$\partial \sigma \ast \hcL(\sigma,\cB) \subset \SpheresNS(M_{n,b},H)$.  In fact, $\hcL(\sigma,\cB)$
is contained in the link of $\sigma$, so $\sigma \ast \hcL(\sigma,\cB) \subset \SpheresNS(M_{n,b},H)$.

Condition (iii).b says that all simplices of $\partial \sigma \ast \hcL(\sigma,\cB)$ that are in $\cB$
lie in $\partial \sigma$.  Again, even more is true: all simplices of $\sigma \ast \hcL(\sigma,\cB)$ that
are in $\cB$ lie in $\sigma$.

Finally, Condition (iii).c says that $\hcL(\sigma,\cB)$ has the disc local injectivity
property up to dimension
\[d - \dim(\sigma) = (n-1-\rank(H)) - \dim(\sigma).\]
Let $X_i \cong M_{n_i,b_i}$.  Theorem \ref{theorem:nspherescm} says that $\NSpheres(X_1,H)$
is weakly Cohen--Macaulay of dimension $n_1-1-\rank(H)$, so by Lemma \ref{lemma:makeinjectivecm}
it has the disc local injectivity property up to dimension $n_1-2-\rank(H)$.  Similarly,
for $2 \leq i \leq r$ the complex $\NSpheres(X_i)$ has the disc local injectivity property
up to dimension $n_i-2$.  Applying Lemma \ref{lemma:discjoin}, we see that
\[\hcL(\sigma,\cB) = \cL(\sigma,\cB) = \NSpheres(X_1,H) \ast \NSpheres(X_2) \ast \cdots \ast \NSpheres(X_r)\]
has the disc local injectivity property up to dimension
\[(n_1-2-\rank(H)) + (n_2-2) + \cdots + (n_r-2) + 2(r-1) = (n_1+\cdots+n_r) - 2 - \rank(H).\]
We must prove that this is at least $n-1-\rank(H) - \dim(\sigma)$, i.e., that
\[n \leq n_1 + \cdots + n_r - 1 + \dim(\sigma).\]
Letting $\Gamma(\sigma)$ be the dual graph of $\sigma$, Lemma \ref{lemma:dualgraph} says that
\[n = \rank(\pi_1(\Gamma(\sigma))) + n_1 + \cdots + n_r.\]
Thus reduces us to proving that
\[\rank(\pi_1(\Gamma(\sigma))) \leq \dim(\sigma) - 1.\]
The graph $\Gamma(\sigma)$ has $r$ vertices and $\dim(\sigma)+1$ edges.  It follows that
\[1-\rank(\pi_1(\Gamma)) = r - \dim(\sigma) - 1 \quad \text{and hence} \quad \rank(\pi_1(\Gamma)) = 2+\dim(\sigma)-r.\]
This reduces our desired inequality to
\[2 + \dim(\sigma) - r \leq \dim(\sigma) - 1,\]
i.e., $r \geq 3$.  This is exactly the defining criterion for simplices in $\cB$: they must have at least $3$
components in their complement.  The lemma follows.
\end{proof}

\section{The augmented nonseparating complex of spheres \texorpdfstring{$\ANSpheres(M_{n,b},H)$}{AN(M,H)} III: disc local injectivity}
\label{section:augmenteddisc}

We now prove disc local injectivity for the augmented nonseparating complex of spheres $\ANSpheres(M_{n,b},H)$.

\subsection{Filtration}
Our goal is to deduce this from disc local injectivity of the expanded complex $\ANSpheresEX(M_{n,b},H)$.  
To do this, we consider a sequence of complexes that interpolate between $\ANSpheresEX(M_{n,b},H)$ and
$\ANSpheres(M_{n,b},H)$:

\begin{definition}
For some $n \geq 0$ and $b \geq 1$ and $m \geq 0$, fix a basepoint $\ast \in \partial M_{n,b}$
and let $H < \pi_1(M_{n,b},\ast) \cong F_n$ be a free factor.  The {\em rank-$m$ expanded augmented nonseparating complex of
$H$-compatible spheres} in $M_{n,b}$, denoted $\ANSpheresEX^m(M_{n,b},H)$, is
the subcomplex of $\Spheres(M_{n,b},H)$ whose $k$-simplices are $H$-compatible rank-$k$ sphere
systems $\sigma = \{S_0,\ldots,S_k\}$ in $M_{n,b}$ with the following properties:
\begin{itemize}
\item Each $S_i$ is a nonseparating sphere.
\item Let $\sigma'$ be the separating core of $\sigma$.  Then either $\sigma'$ is empty (so
$\sigma$ does not separate $M_{n,b}$), or $\sigma'$ has two components $X$ and $Y$ in
its complement.  Moreover, in the latter case if $X$ is the basepoint-containing component
then $\rank(\pi_1(Y)) \leq m$.\qedhere
\end{itemize}
\end{definition}

We have the following:

\begin{lemma}
\label{lemma:identifyan}
For some $n \geq 0$ and $b \geq 1$, fix a basepoint $\ast \in \partial M_{n,b}$
and let $H < \pi_1(M_{n,b},\ast) \cong F_n$ be a free factor.  Then for all
simplices $\sigma$ of $\ANSpheresEX(M_{n,b},H)$ and all components $Z$ of
the complement of $\sigma$, we have $\rank(\pi_1(Z)) \leq n-1$.  In particular,
$\ANSpheresEX^{n-1}(M_{n,b},H) = \ANSpheresEX(M_{n,b},H)$.
\end{lemma}
\begin{proof}
Since all vertices of $\sigma$ are nonseparating spheres, removing the edge of the dual
graph $\Gamma(\sigma)$ corresponding to any one of them does not separate $\Gamma(\sigma)$.
This implies that $\Gamma(\sigma)$ has a nontrivial fundamental group.  The
lemma now follows from Lemma \ref{lemma:dualgraph}.
\end{proof}

It follows from this that we have
\begin{align*}
\ANSpheres(M_{n,b},H) &= \ANSpheresEX^0(M_{n,b},H) \subset \ANSpheresEX^1(M_{n,b},H) \\
                      &\quad \quad \quad \subset \cdots \subset \ANSpheresEX^{n-1}(M_{n,b},H) = \ANSpheresEX(M_{n,b},H).
\end{align*}

\subsection{Disc local injectivity}
Our result is as follows.  The case $m=0$ of it completes the proof of Theorem \ref{theorem:anspheresbetter}.

\begin{lemma}
\label{lemma:augdisc}
For some $n,b \geq 1$ and $0 \leq m \leq n-1$, fix a basepoint $\ast \in \partial M_{n,b}$
and let $H < \pi_1(M_{n,b},\ast) \cong F_n$ be a free factor.  Assume
that $\rank(H) \leq n-1$ and that if $\rank(H) = n-1$, then $b \geq 2$.
Then $\ANSpheresEX^m(M_{n,b},H)$ has
the disc local injectivity property up to dimension $(n-1-\rank(H))$.
\end{lemma}
\begin{proof}
The proof is by reverse induction on $m$.  Lemma \ref{lemma:identifyan} says that
$\ANSpheresEX^{n-1}(M_{n,b},H) = \ANSpheresEX(M_{n,b},H)$, so
the base case $m=n-1$ is provided by Lemma \ref{lemma:anspheresurdisc}.  Assume, therefore, that
$0 \leq m < n-1$ and that $\ANSpheresEX^{m+1}(M_{n,b},H)$ has
the disc local injectivity property up to dimension $d=n-1-\rank(H)$.
To prove the same for
$\ANSpheresEX^m(M_{n,b},H)$, it is enough to verify the conditions of
Proposition \ref{proposition:avoidbadhard} for the inclusion
\[\ANSpheresEX^m(M_{n,b},H) \hookrightarrow \ANSpheresEX^{m+1}(M_{n,b},H).\]
The input to Proposition \ref{proposition:avoidbadhard} is a set $\cB$ of ``bad simplices''
of $\ANSpheresEX^{m+1}(M_{n,b},H)$, which
for us will be the set of simplices $\sigma$ satisfying the following:
\begin{itemize}
\item $\sigma$ separates $M_{n,b}$ into two components $X$ and $Y$, and
\item $\sigma$ equals its own separating core, and
\item if $X$ is the basepoint-containing component of the complement, then $\rank(\pi_1(Y)) = m+1$.
\end{itemize}
We now verify each hypothesis of Proposition \ref{proposition:avoidbadhard} in turn.

Condition (i) says that a simplex $\sigma$ of $\ANSpheresEX^{m+1}(M_{n,b},H)$ lies in $\ANSpheresEX^{m}(M_{n,b},H)$ if
and only if none of its faces are in $\cB$.  Since no simplices of $\cB$ lie in $\ANSpheresEX^{m}(M_{n,b},H)$, if
a face of $\sigma$ lies in $\cB$ then 
$\sigma$ does not lie in $\ANSpheresEX^{m}(M_{n,b},H)$.  Conversely, assume that no face of $\sigma$ lies in $\cB$.
We will prove that $\sigma$ lies in $\ANSpheresEX^{m}(M_{n,b},H)$.
If $\sigma$ does not separate $M_{n,b}$, then it trivially
lies in $\ANSpheresEX^{m}(M_{n,b},H)$.  If $\sigma$ does separate, then let $\sigma'$
be its separating core.  Let $X$ and $Y$ be the components
of the complement of $\sigma'$, with $X$ the basepoint-containing component.  Since $\sigma$ is a
simplex of $\ANSpheresEX^{m+1}(M_{n,b},H)$, we have $\rank(\pi_1(Y)) \leq m+1$.  Moreover,
since $\sigma'$ does not lie in $\cB$ this must be a strict inequality, i.e., $\rank(\pi_1(Y)) \leq m$.
It follows that $\sigma$ lies in $\ANSpheresEX^{m}(M_{n,b},H)$, as desired.

Condition (ii) says that if $\sigma,\sigma' \in \cB$ are such that $\sigma \cup \sigma'$ is a simplex
of $\ANSpheresEX^{m+1}(M_{n,b},H)$, then $\sigma \cup \sigma' \in \cB$.  In fact, since
simplices of $\ANSpheresEX^{m+1}(M_{n,b},H)$ can separate $M_{n,b}$ into at most $2$ components, this
can only happen if $\sigma = \sigma'$, so there is nothing to prove.

Condition (iii) says that for all $\sigma \in \cB$, there is a subcomplex $\hcL(\sigma,\cB)$ of
$\ANSpheresEX^{m+1}(M_{n,b},H)$ with $\cL(\sigma,\cB) \subset \hcL(\sigma,\cB)$ that satisfies conditions
(iii).a--(iii).c.  Here $\cL(\sigma,\cB)$ is the complex defined in Definition \ref{definition:badlink}.
Fixing some $\sigma \in \cB$, before we can define $\hcL(\sigma,\cB)$ we must give a concrete
description of $\cL(\sigma,\cB)$.  Let $X$ and $Y$ be the components of the complement of
$\sigma$, with $X$ the basepoint-containing component.  The group $\pi_1(Y)$ thus has rank $m+1$.
It follows from the definitions that
\[\cL(\sigma,\cB) = \Link_{\ANSpheresEX^{m+1}(M_{n,b},H)}(\sigma) = \NSpheres(X,H) \ast \NSpheres(Y).\]
We then define
\[\hcL(\sigma,\cB) = \NSpheres(X,H) \ast \ANSpheresEX(Y) = \NSpheres(X,H) \ast \ANSpheresEX^m(Y),\]
where the final equality follows from Lemma \ref{lemma:identifyan}.
We verify (iii).a--(iii).c for this as follows.

Condition (iii).a says that
$\partial \sigma \ast \hcL(\sigma,\cB) \subset \ANSpheresEX^{m+1}(M_{n,b},H)$.  This follows
immediately from the fact that proper faces of $\sigma$ do not separate $M_{n,b}$.

Condition (iii).b says that all simplices of $\partial \sigma \ast \hcL(\sigma,\cB)$ that are in $\cB$
lie in $\partial \sigma$.  Even more is true: since proper faces of $\sigma$ do not separate
$M_{n,b}$, no simplices of $\partial \sigma \ast \hcL(\sigma,\cB)$ lie in $\cB$.

Finally, Condition (iii).c says that $\hcL(\sigma,\cB)$ has the disc local injectivity
property up to dimension
\[d - \dim(\sigma) = (n-1-\rank(H)) - \dim(\sigma).\]
Let $X \cong M_{n',b'}$ and $Y \cong M_{n'',b''}$, so by definition $n'' = m+1$.  It follows from Lemma \ref{lemma:anspheresurdisc}
that
\[\ANSpheresEX^m(Y) = \ANSpheresEX(Y)\]
has the disc local injectivity property up to
dimension $n''-1$.  Also, Theorem \ref{theorem:nspherescm} says that $\NSpheres(X,H)$
is weakly Cohen--Macaulay of dimension $n'-1-\rank(H)$, so by Lemma \ref{lemma:makeinjectivecm}
it has the disc local injectivity property up to dimension $n'-2-\rank(H)$.  
Applying Lemma \ref{lemma:discjoin}, we see that
\[\hcL(\sigma,\cB) = \NSpheres(X,H) \ast \ANSpheresEX^m(Y)\]
has the disc local injectivity property up to dimension
\[(n'-2-\rank(H)) + (n''-1)+2 = (n'+n'') - 1 -\rank(H).\]
We must prove that this is at least $n-1-\rank(H)-\dim(\sigma)$, i.e., that
\[n'+n'' \geq n - \dim(\sigma).\]
In fact this is an equality.  To see this note that 
the dual graph $\Gamma(\sigma)$ has two vertices corresponding to $X$ and $Y$ and
$\dim(\sigma)+1$ edges.  It follows that $\rank(\pi_1(\Gamma(\sigma))) = \dim(\sigma)$.  Applying
Lemma \ref{lemma:dualgraph}, we see that
\[n = \rank(\pi_1(\Gamma(\sigma))) + n' + n'' = \dim(\sigma) + n' + n''.\]
The lemma follows.
\end{proof}

\section{A presentation of the Steinberg module for \texorpdfstring{$\Aut(F_n)$}{Aut(Fn)}}
\label{section:presentation}

In this section, we use our high connectivity results for the augmented nonseparating
sphere complex to construct a presentation for the Steinberg module $\St(F_n)$, whose
definition we will recall below.

\subsection{Generalities about posets}

Let $\bA$ be a poset.  Let $\bA^{\opp}$ be the opposite poset of $\bA$.  Recall that the
geometric realization $|\bA|$ is the simplicial complex whose
$k$-simplices are chains $a_0 < \cdots < a_k$ in $\bA$.  The
simplicial complexes $|\bA|$ and $|\bA^{\opp}|$ are isomorphic.  A key
example is as follows:

\begin{definition}
Let $X$ be a simplicial complex.  The {\em poset of simplices} of $X$, denoted $\Poset(X)$, is the
poset whose elements are simplices of $X$ ordered by inclusion.
\end{definition}

For a simplicial complex $X$, the geometric realization $|\Poset(X)|$ is the first barycentric
subdivision of $X$.  We will also need the following two notions:

\begin{definition}
Let $\bA$ be a poset and $a \in \bA$. The {\em height} of $a$, denoted $\height(a)$, 
is the maximal $k \geq 0$ such that there exists a chain $a_0 < a_1 < \cdots < a_k = a$.
\end{definition}

\begin{example}
If $X$ is a simplicial complex and $\sigma \in \Poset(X)$, then $\height(\sigma) = \dim(\sigma)$.
\end{example}

\begin{definition}
Let $\phi \colon \bA \rightarrow \bB$ be a map of posets and $b \in \bB$.  The {\em poset fiber} of $\phi$ over $b$, denoted 
$\phi^{\leq b}$, is the subposet of $\bA$ consisting of all $a \in \bA$ such that $\phi(a) \leq b$.
\end{definition}

The following is the case $m=1$ of Church--Putman \cite[Proposition 2.3]{ChurchPutmanCodim1}, which
builds on work of Quillen \cite{QuillenPoset}.  Recall from Definition \ref{definition:cm} that
a simplicial complex is Cohen--Macaulay of dimension $d$ if it is $d$-dimensional and weakly Cohen--Macaulay of
dimension $d$.

\begin{proposition}[{Church--Putman \cite[Proposition 2.3]{ChurchPutmanCodim1}}]
\label{proposition:posetfiber}
Fix an abelian group $R$.  Let $\phi\colon \bA \rightarrow \bB$ be a map of posets.  Assume that $|\bB|$ is 
Cohen--Macaulay of dimension $d \geq 0$ and that for all $b \in \bB$, the geometric realization of the poset fiber
$|\phi^{\leq b}|$ is $\height(b)$-connected.  Then
$\phi_{\ast}\colon \RH_i(|\bA|;R) \rightarrow \RH_i(|\bB|;R)$ is an isomorphism for $i \leq d$.
\end{proposition}

\subsection{Relating the free factor complex to the sphere complex}
Recall from the introduction that the free factor complex $\Tits(F_n)$ is the geometric
realization of the following poset:

\begin{definition}
Let $\BTits(F_n)$ be the poset of proper nontrivial free factors of $F_n$, ordered
by inclusion.
\end{definition}

Hatcher--Vogtmann \cite{HatcherVogtmann, HatcherVogtmannRevised} proved that $\Tits(F_n) = |\BTits(F_n)|$ is homotopy
equivalent to a wedge of $(n-2)$-dimensional spheres, and by definition
\[\St(F_n) = \RH_{n-2}(\Tits(F_n);\Q) = \RH_{n-2}(|\BTits(F_n)|;\Q).\]
We will relate this to the sphere complex $\Spheres(M_{n,1})$ and its subcomplexes.  Fix $\ast \in \partial M_{n,1}$.  
Recall from the introduction that the action of
$\Diff^{+}(M_{n,1},\partial M_{n,1})$ on $\Spheres(M_{n,1})$ factors through
\[\Aut(\pi_1(M_{n,1},\ast)) \cong \Aut(F_n).\]
We thus get an action of $\Aut(F_n)$ on $\Spheres(M_{n,1})$.  The group $\Aut(F_n)$
also acts on $\St(F_n)$, and while relating $\St(F_n)$ to $\Spheres(M_{n,1})$ and
its subcomplexes we will keep track of these group actions.

We introduce the following subcomplex of the augmented nonseparating sphere complex $\ANSpheres(M_{n,1})$. 

\begin{definition}
For $n \geq 0$, let $\ANSpheres'(M_{n,1})$ be the subcomplex of $\ANSpheres(M_{n,1})$ consisting of
simplices $\sigma$ such that the basepoint-containing component of the
complement of $\sigma$ has a nontrivial fundamental group. 
\end{definition}

Having done this, we can now make the following key definition:

\begin{definition}
For some $n \geq 0$, fix a basepoint $\ast \in \partial M_{n,1}$.  Identify $\pi_1(M_{n,1},\ast)$ with $F_n$.
The {\em complement map} is the map of posets $\Upsilon\colon \Poset(\ANSpheres'(M_{n,1})) \rightarrow \BTits(F_n)^{\opp}$
defined as follows.  Consider $\sigma \in \Poset(\ANSpheres'(M_{n,1}))$, and let $X$ be the basepoint-containing
component of the complement of $\sigma$.  Then
\[\Upsilon(\sigma) = \pi_1(X,\ast) \subset \pi_1(M_{n,1},\ast) = F_n.\qedhere\]
\end{definition}

\begin{remark}
This definition makes sense since $\sigma$ is a simplex of $\ANSpheres'(M_{n,1})$, so the basepoint-containing
component $X$ of the complement of $\sigma$ is not simply connected and $\Upsilon(\sigma) = \pi_1(X,\ast)$ is
not trivial.  The map $\Upsilon$ cannot be defined on the whole complex $\ANSpheres(M_{n,1})$.
\end{remark}

\begin{remark}
The target of the complement map is the opposite poset $\BTits(F_n)^{\opp}$ because if $\sigma$ and $\sigma'$
are simplices of $\ANSpheres'(M_{n,1})$ with $\sigma \subset \sigma'$, then the basepoint-containing 
component $X$ of the complement of $\sigma$ contains the basepoint-containing component $X'$ of the complement
of $\sigma'$, so
\[\Upsilon(\sigma) = \pi_1(X,\ast) \supset \pi_1(X',\ast) = \Upsilon(\sigma').\qedhere.\]
\end{remark}

\begin{remark}
\label{remark:equivariance1}
The action of $\Aut(F_n)$ on $\Spheres(M_{n,1})$ preserves the subcomplexes $\ANSpheres(M_{n,1})$ and
$\ANSpheres'(M_{n,1})$, and the map $\Upsilon\colon \Poset(\ANSpheres'(M_{n,1})) \rightarrow \BTits(F_n)^{\opp}$
is $\Aut(F_n)$-equivariant.
\end{remark}

Identify 
\[\RH_{n-2}(\ANSpheres'(M_{n,1});\Q) \quad \text{with} \quad \RH_{n-2}(|\Poset(\ANSpheres'(M_{n,1}))|;\Q)\]
and
\[\St(F_n) = \RH_{n-2}(\Tits(F_n);\Q) \quad \text{with} \quad \RH_{n-2}(|\BTits(F_n)|;\Q) = \RH_{n-2}(|\BTits(F_n)^{\opp}|;\Q).\]  
We then have the following:

\begin{lemma}
\label{lemma:spaniso}
For some $n \geq 2$, fix a basepoint $\ast \in \partial M_{n,1}$.  Identify $\pi_1(M_{n,1},\ast)$ with $F_n$. 
Then the complement map 
$\Upsilon\colon \Poset(\ANSpheres'(M_{n,1})) \rightarrow \BTits(F_n)^{\opp}$ induces an $\Aut(F_n)$-equivariant isomorphism
\[\Upsilon_{\ast}\colon \RH_{n-2}(\ANSpheres'(M_{n,1});\Q) \stackrel{\cong}{\longrightarrow} \St(F_n).\]
\end{lemma}
\begin{proof}
Since $\Upsilon$ is $\Aut(F_n)$-equivariant (see Remark \ref{remark:equivariance1}), the map $\Upsilon_{\ast}$
is $\Aut(F_n)$-equivariant.
To prove that it is an isomorphism, it is enough to show that the complement map
\[\Upsilon\colon \Poset(\ANSpheres'(M_{n,1})) \rightarrow \BTits(F_n)^{\opp}\]
satisfies the conditions of Proposition \ref{proposition:posetfiber} with $d = n-2$.
The first of these hypotheses is that $|\BTits(F_n)^{\opp}|$ is Cohen--Macaulay of dimension $(n-2)$,
which was proved\footnote{They actually proved that $|\Tits(F_n)|$ is Cohen--Macaulay of dimension $(n-2)$,
but the geometric realization of a poset is isomorphic to the geometric realization of its opposite poset.}
by Hatcher--Vogtmann \cite[\S 4]{HatcherVogtmann, HatcherVogtmannRevised}.  The second hypothesis is that for all $H \in \BTits(F_n)^{\opp}$,
the poset fiber $\Upsilon^{\leq H}$ is $\height(H)$-connected.  This poset fiber is the poset of simplices
$\sigma$ of $\ANSpheres'(M_{n,1})$ such that the fundamental group of the basepoint-containing
component of the complement contains $H$.  By definition, this is precisely the poset of simplices
of $\ANSpheres(M_{n,1},H)$ -- note that there is no $'$ in this since $H \neq 1$, so simplices
of $\ANSpheres(M_{n,1},H)$ automatically have a non-simply-connected basepoint-containing component
of their complement.  Theorem \ref{theorem:nspheresredcon} says that $\ANSpheres(M_{n,1},H)$ and hence
its poset of simplices $\Poset(\ANSpheres(M_{n,1},H))$ is $(n-1-\rank(H))$ connected.  Since we are regarding
$H$ as an element of the opposite poset $\BTits(F_n)^{\opp}$, we have $\height(H) = n-1-\rank(H)$,
so the lemma follows.
\end{proof} 

\subsection{Relative homology}
We now study the topology of $\ANSpheres'(M_{n,1})$.  The starting point is the following:

\begin{lemma}
\label{lemma:identifyrel}
For $n \geq 2$, we have an $\Aut(F_n)$-equivariant isomorphism
\[\HH_{n-1}(\ANSpheres(M_{n,1}),\ANSpheres'(M_{n,1});\Q) \cong \RH_{n-2}(\ANSpheres'(M_{n,1});\Q).\]
\end{lemma}
\begin{proof}
In this proof, all homology has $\Q$-coefficients, though we remark that $\Q$ can be replaced
by any abelian group.  The long exact sequence of the pair $(\ANSpheres(M_{n,1}),\ANSpheres'(M_{n,1}))$ contains the segment
\begin{align*}
\RH_{n-1}(\ANSpheres(M_{n,1})) &\rightarrow \HH_{n-1}(\ANSpheres(M_{n,1}),\ANSpheres'(M_{n,1}))\\
&\quad\quad\quad \stackrel{\partial}{\rightarrow} \RH_{n-2}(\ANSpheres'(M_{n,1}))
\rightarrow \RH_{n-2}(\ANSpheres(M_{n,1})).
\end{align*}
Since $n \geq 2$, Theorem \ref{theorem:nspheresredcon} implies that
\[\RH_{n-1}(\ANSpheres(M_{n,1})) = \RH_{n-2}(\ANSpheres(M_{n,1})) = 0.\]
The lemma follows.
\end{proof}

To understand this relative homology group, the key result is as follows.
Recall that $\ANSpheres(M_{n,1})$ is $n$-dimensional.

\begin{lemma}
\label{lemma:subcomplex}
Fix some $n \geq 2$.  The following then hold:
\begin{itemize}
\item $\ANSpheres'(M_{n,1})$ is $(n-1)$-dimensional.
\item The $(n-2)$-skeletons of $\ANSpheres'(M_{n,1})$ and $\ANSpheres(M_{n,1})$ are equal.
\item The only $(n-1)$-simplices of $\ANSpheres(M_{n,1})$ that do not lie in $\ANSpheres'(M_{n,1})$ are the $(n-1)$-simplices
of $\NSpheres(M_{n,1})$.
\end{itemize}
\end{lemma}
\begin{proof}
Since $\ANSpheres'(M_{n,1})$ is a subcomplex of $\ANSpheres(M_{n,1})$, it is enough to characterize which simplices
of $\ANSpheres(M_{n,1})$ lie in $\ANSpheres'(M_{n,1})$.  Let $\sigma$ be a $k$-simplex of $\ANSpheres(M_{n,1})$.  
Write $\sigma = \{S_0,\ldots,S_k\}$.  Assume first that the union of the $S_i$ does not separate $M_{n,1}$, so $\sigma$ is a simplex of $\NSpheres(M_{n,1})$.
We thus have $k \leq n-1$, and the only component of the complement of $\sigma$ is homeomorphic to
$M_{n-(k+1),1+2(k+1)}$.  This is simply-connected if and only if $k = n-1$, so we see that $\sigma$ lies in
$\ANSpheres'(M_{n,1})$ if and only if $k \leq n-2$, as claimed.

Assume next that the union of the $S_i$ does separate $M_{n,1}$.  Since $\sigma$ is a simplex of
$\ANSpheres(M_{n,1})$, there are two components $X$ and $Y$ of the complement of $\sigma$.  Moreover,
letting $X$ be the basepoint-containing component we have $\pi_1(Y) = 1$.  
The dual graph $\Gamma(\sigma)$ has two vertices and $(k+1)$ edges, so its fundamental group
is $F_k$.  Letting $m$ be such that $\pi_1(X,\ast) \cong F_m$, Lemma \ref{lemma:dualgraph} says that
\[n = \pi_1(\Gamma(\sigma)) + m + 0 = k + m.\]
The simplex $\sigma$ lies in $\ANSpheres'(M_{n,1})$ if and only if $m \geq 1$, and the above equality
shows that this holds if and only if $k \leq n-1$.  The lemma follows.
\end{proof}

This has the following consequence:

\begin{lemma}
\label{lemma:identifychains}
For $n \geq 2$ and $k \geq 0$, we have $\Aut(F_n)$-equivariant isomorphisms
\[\CC_k(\ANSpheres(M_{n,1}),\ANSpheres'(M_{n,1});\Q) = \begin{cases}
\CC_{n-1}(\NSpheres(M_{n,1});\Q) & \text{if $k=n-1$},\\
\CC_{n}(\ANSpheres(M_{n,1});\Q)  & \text{if $i=n$}, \\
0  & \text{otherwise}.\end{cases}\]
\end{lemma}
\begin{proof}
Since $\ANSpheres(M_{n,1})$ is $n$-dimensional, we have
\[\CC_k(\ANSpheres(M_{n,1}),\ANSpheres'(M_{n,1});\Q) = 0 \quad \text{if $k \geq n+1$}.\]
The remainder of the lemma follows from Lemma \ref{lemma:subcomplex}.
\end{proof}

Combining these results we obtain the following presentation of $\St(F_n)$.

\begin{theorem} 
\label{theorem:stpresentation}
For $n \geq 2$, we have an exact sequence of $\Aut(F_n)$-modules
\[\CC_n(\ANSpheres(M_{n,1});\Q) \to \CC_{n-1}(\NSpheres(M_{n,1});\Q) \to \St(F_n) \to 0.\]
\end{theorem}
\begin{proof}
In this proof, all homology has $\Q$-coefficients.  Combining Lemmas \ref{lemma:spaniso} and \ref{lemma:identifyrel}, we get an 
$\Aut(F_n)$-equivariant isomorphism
\[\HH_{n-1}(\ANSpheres(M_{n,1}),\ANSpheres'(M_{n,1})) \cong \St(F_n).\]
To compute this relative homology group, consider the relevant terms in the simplicial chain complex:
\begin{align*}
\CC_n(\ANSpheres(M_{n,1}),\ANSpheres'(M_{n,1})) \longrightarrow &\CC_{n-1}(\ANSpheres(M_{n,1}),\ANSpheres'(M_{n,1})) \\
&\quad\quad\quad \quad \quad \longrightarrow \CC_{n-2}(\ANSpheres(M_{n,1}),\ANSpheres'(M_{n,1})).
\end{align*}
Lemma \ref{lemma:identifychains} shows that these terms are equal to
\[\CC_n(\ANSpheres(M_{n,1})) \rightarrow \CC_{n-1}(\NSpheres(M_{n,1})) \rightarrow 0.\]
The theorem follows.
\end{proof}

\subsection{Flatness}

The following lemma shows that this presentation can be used to compute the low-dimensional rational homology
of $\Aut(F_n)$ with coefficients in $\St(F_n)$:

\begin{lemma}
\label{lemma:flatness}
Fix some $n \geq 2$.  Then $\CC_n(\ANSpheres(M_{n,1});\Q)$ and $\CC_{n-1}(\NSpheres(M_{n,1});\Q)$ are
flat $\Z[\Aut(F_n)]$-modules.
\end{lemma}
\begin{proof}
It is enough to prove that if $\sigma$ is either an $n$-simplex of $\ANSpheres(M_{n,1})$ or an $(n-1)$-simplex
of $\NSpheres(M_{n,1})$, then the $\Aut(F_n)$-stabilizer of $\sigma$ is finite.  See, e.g., \cite[Lemma 3.2]{ChurchFarbPutmanVanish}.\footnote{The key
ingredients in the proof are Shapiro's Lemma and the fact that finite groups have trivial homology in positive degrees with respect
to coefficients that are vector spaces over fields of characteristic $0$.}  Since the action of
the mapping class group $\Mod(M_{n,1})$ on these simplicial complexes factors through $\Aut(F_n)$, this
will follow if we can prove that the $\Mod(M_{n,1})$-stabilizer $G_{\sigma}$ of $\sigma$ is finite.

Since every $n$-simplex of $\ANSpheres(M_{n,1})$ contains an $(n-1)$-simplex of $\NSpheres(M_{n,1})$, it
is enough to deal with the case where $\sigma$ is an $(n-1)$-simplex of $\NSpheres(M_{n,1})$.  Write
$\sigma = \{S_0,\ldots,S_{n-1}\}$.  Define $G'_{\sigma}$ to be the subgroup of $\Mod(M_{n,1})$ consisting
of mapping classes $f$ satisfying the following condition:
\begin{itemize}
\item For each $0 \leq i \leq n-1$, the mapping class $f$ can be realized by a diffeomorphism of $M_{n,1}$ that
fixes a tubular neighborhood of $S_i$.
\end{itemize}
The group $G'_{\sigma}$ is a finite-index subgroup\footnote{This uses the fact that homotopy implies ambient isotopy for
collections of pairwise non-homotopic disjoint spheres in $M_{n,1}$, which was proved by Laudenbach \cite{LaudenbachPaper, LaudenbachBook}.
The differences between $G_{\sigma}$ and $G'_{\sigma}$ are that
elements of $G_{\sigma}$ can permute the $S_i$ and also ``flip them around'', i.e., reverse the directions of arcs
transverse to them.} of $G_{\sigma}$, so it is enough to prove that $G'_{\sigma}$ is finite.  Cutting $M_{n,1}$ open
along the $S_i$ yields $M_{0,2n+1}$.  Reversing this, there is a homomorphism $\psi\colon \Mod(M_{0,2n+1}) \rightarrow \Mod(M_{n,1})$
induced by the map taking diffeomorphisms of $M_{0,2n+1}$ that are the identity on $\partial M_{0,2n+1}$ and gluing boundary
components together in pairs.  The image of $\psi$ is $G'_{\sigma}$, so we are reduced to proving that $\Mod(M_{0,2n+1})$
is a finite group.  

In fact, Laudenbach (\cite{LaudenbachPaper, LaudenbachBook}; see \cite{Landgraf} for an alternate approach using a
``Birman exact sequence'') proved that $\Mod(M_{0,2n+1}) \cong (\Z/2)^{2n}$, generated by sphere twists about $2n$ of
the boundary components.  There is no need to take the sphere twist about the remaining boundary component since the product
of sphere twists about all the boundary components of $M_{0,2n+1}$ is trivial; see \cite[p.\ 214--215]{HatcherWahl3Manifolds} for
a simple proof of this.  The lemma follows.
\end{proof}

\section{Proof of main theorem}
\label{section:mainproof}

We close the paper by finally proving Theorem \ref{maintheorem:steinberghomology} which asserts
that $\HH_k(\Aut(F_n);\St(F_n)) = 0$ for $k \in \{0,1\}$ and $n \geq 2$.

\begin{proof}[Proof of Theorem \ref{maintheorem:steinberghomology}]
Let $n \geq 2$, and consider the presentation of $\St(F_n)$ from Theorem \ref{theorem:stpresentation}:
\[\CC_n(\ANSpheres(M_{n,1});\Q) \to \CC_{n-1}(\NSpheres(M_{n,1});\Q) \to \St(F_n) \to 0.\]
By Lemma \ref{lemma:flatness}, both $\CC_n(\ANSpheres(M_{n,1});\Q)$ and $\CC_{n-1}(\NSpheres(M_{n,1});\Q)$ are flat
$\Z[\Aut(F_n)]$-modules.  Complete this to a flat resolution of $\St(F_n)$:
\[\cdots \rightarrow \cF_2 \rightarrow \CC_n(\ANSpheres(M_{n,1});\Q) \rightarrow \CC_{n-1}(\NSpheres(M_{n,1});\Q) \rightarrow \St(F_n) \rightarrow 0.\]
For a group $G$ and a $\Z[G]$-module $M$, write $M_G$ for the $G$-coinvariants of $M$.  We can use our flat
resolution to compute $\HH_{\bullet}(\Aut(F_n);\St(F_n))$, which is the homology of the chain complex
\[\cdots \rightarrow (\cF_2)_{\Aut(F_n)} \rightarrow (\CC_n(\ANSpheres(M_{n,1});\Q))_{\Aut(F_n)} \rightarrow (\CC_{n-1}(\NSpheres(M_{n,1});\Q)_{\Aut(F_n)} \rightarrow 0.\]
To prove the theorem, it suffices to prove that
\[(\CC_n(\ANSpheres(M_{n,1});\Q))_{\Aut(F_n)} = (\CC_{n-1}(\NSpheres(M_{n,1});\Q)_{\Aut(F_n)} = 0.\]
The vector space $\CC_n(\ANSpheres(M_{n,1});\Q)$ (resp.\ $\CC_{n-1}(\NSpheres(M_{n,1});\Q)$)
is spanned by oriented $n$-simplices of $\ANSpheres(M_{n,1})$ (resp.\ $(n-1)$-simplices of $\NSpheres(M_{n,1})$).
For such an oriented simplex $\sigma$, we will prove below in Lemma \ref{lemma:fliplemma} that there exists
some $f \in \Aut(F_n)$ such that $f(\sigma)$ equals $\sigma$, but with the opposite orientation.  This will imply
that the images of $\sigma$ and $-\sigma$ in the $\Aut(F_n)$-coinvariants are equal, and thus that the image
of $2\sigma$ in the $\Aut(F_n)$-coinvariants is $0$.  Since we are working over $\Q$, this implies that the
image of $\sigma$ in the $\Aut(F_n)$-coinvariants is $0$, so these coinvariants are $0$, as desired.
\end{proof}

\begin{lemma} 
\label{lemma:fliplemma}
Fix some $n \geq 2$.  Let $\sigma$ be either an oriented $n$-simplex of $\ANSpheres(M_{n,1})$ or an oriented $(n-1)$-simplex
of $\NSpheres(M_{n,1})$.  Then there exists some $f \in \Aut(F_n)$ such that $f(\sigma)$ equals $\sigma$, but with
the opposite orientation.
\end{lemma}
\begin{proof}
For $k \geq 2$, we will prove that this holds more generally for all oriented $(k-1)$-simplices $\sigma$ of the complex $\ANSpheres(M_{n,1})$.
The action of the mapping class group $\Mod(M_{n,1})$ on $\ANSpheres(M_{n,1})$ factors through $\Aut(F_n)$, so it is enough
to find some $\phi \in \Mod(M_{n,1})$ such that $\phi(\sigma)$ equals $\sigma$, but with the opposite orientation.
Let $\sigma = \{S_1,\ldots,S_k\}$, ordered in a way compatible with the orientation of $\sigma$.

There are either one or two components in the complement of $\sigma$.  Assume first that there is one component in the complement, which
we call $X$.  We then have that $X \cong M_{n-k,2k+1}$.  Enumerate the components of $\partial X$ as 
$\{\partial, \beta_1, \beta_1', \ldots, \beta_k, \beta_k'\}$, where $\partial$ is the component of $\partial M_{n,1}$ and
$\beta_i$ and $\beta_i'$ are the two components that glue together to form $S_i$.  We can then find an orientation-preserving
diffeomorphism $\psi\colon X \rightarrow X$ such that
\[\psi(\beta_1) = \beta_2 \quad \text{and} \quad \psi(\beta_2) = \beta_1 \quad \text{and} \quad 
\psi(\beta_1') = \beta_2' \quad \text{and} \quad \psi(\beta_2') = \beta_1'\]
and such that
\[\psi|_{\partial} = \psi|_{\beta_i} = \psi|_{\beta'_i} = \text{id} \quad \text{for $3 \leq i \leq k$}.\]
Isotoping $\psi$ if necessary to make sure its behavior on $\beta_1 \cup \beta'_1$ matches up with its
behavior on $\beta_2 \cup \beta'_2$, we can glue the boundary components of $X$ back together and get
from $\psi$ an induced diffeomorphism $\phi\colon M_{n,1} \rightarrow M_{n,1}$ that swaps $S_1$ and $S_2$
while fixing $S_i$ for $3 \leq i \leq k$.  It follows that $\phi$ takes $\sigma$ to $\sigma$ but with the opposite
orientation, as desired.

Assume now that there are two components $X$ and $Y$ in the complement of $\sigma$,
with $X$ the basepoint-containing component.  For some $r \geq 3$ and $h \geq 0$
with $r+h = k$, we have
\[Y \cong M_{0,r} \quad \text{and} \quad X \cong M_{n-k+1,1+r+2h}.\]
The condition $r \geq 3$ holds since all the $S_i$ are nonseparating and pairwise non-isotopic.  Enumerate
the components of $\partial X$ and $\partial Y$ as
\[\{\partial,\beta_1,\ldots,\beta_{r},\delta_1,\delta'_1,\ldots,\delta_{h},\delta'_{h}\}
\quad \text{and} \quad \{\beta_1',\ldots,\beta_{r}'\},\]
respectively, where the enumeration is chosen such that the following hold when $X$
and $Y$ are glued together to form $M_{n,1}$:
\begin{itemize}
\item $\partial$ is the component of $\partial M_{n,1}$.
\item $\beta_i \subset \partial X$ and $\beta_i' \subset \partial Y$ are glued together 
to form a sphere in $\{S_1,\ldots,S_k\}$.
\item $\delta_j \subset \partial X$ and $\delta'_j \subset \partial X$ are glued together
to form a sphere in $\{S_1,\ldots,S_k\}$.
\end{itemize}
Let $S_a$ (resp.\ $S_b$) be the sphere in $\{S_1,\ldots,S_k\}$ that is formed
when $\beta_1$ is glued to $\beta'_1$ (resp.\ $\beta_2$ is glued to $\beta'_2$).  
We can then find orientation-preserving diffeomorphisms $\psi_1\colon X \rightarrow X$
and $\psi_2\colon Y \rightarrow Y$ such that
\[\psi_1(\beta_1) = \beta_2 \quad \text{and} \quad \psi_1(\beta_2) = \beta_1 \quad \text{and} \quad 
\psi_2(\beta_1') = \beta_2' \quad \text{and} \quad \psi_2(\beta_2') = \beta_1'\]
and such that
\[\psi_1|_{\partial} = \psi_1|_{\beta_i} = \psi_1|_{\delta_j} = \psi_1|_{\delta'_j} = \psi_2|_{\beta'_i} = \text{id} \quad \text{for $3 \leq i \leq r$ and $1 \leq j \leq h$}.\]
Isotoping $\psi_1$ and $\psi_2$ if necessary to make sure their behavior on the boundaries match
up, we can glue the boundary components of $X$ and $Y$ back up and get from $\psi_1$ and
$\psi_2$ an induced diffeomorphism $\phi\colon M_{n,1} \rightarrow M_{n,1}$ that
swaps $S_a$ and $S_b$ while fixing $S_i$ for $1 \leq i \leq k$ with $i \neq a,b$.
It follows that $\phi$ takes $\sigma$ to $\sigma$ but with the opposite
orientation, as desired.
\end{proof}


\begin{thebibliography}{99}

\bibitem{BestvinaFeighnDuality}
M. Bestvina\ and\ M. Feighn, The topology at infinity of ${\rm Out}(F_n)$, Invent. Math. 140 (2000), no.~3, 651--692. 

\bibitem{BorelSerre}
A. Borel\ and\ J.-P. Serre, Corners and arithmetic groups, Comment. Math. Helv. 48 (1973), 436--491.

\bibitem{BrendleBroaddusPutmanPunctureSteinberg}
T. Brendle, N. Broaddus, and\ A. Putman, The high-dimensional cohomology of the moduli space of curves with level structures II: punctures and boundary, to appear in Israel J. Math. \arxiv{2003.10913} 

\bibitem{BrendleBroaddusPutman}
T. Brendle, N. Broaddus, and\ A. Putman, The mapping class group of connect sums of $S^2 \times S^1$, to appear in Trans. Amer. Math. Soc. \arxiv{2012.01529}

\bibitem{BruckGupta}
B. Br\"{u}ck\ and\ R. Gupta, Homotopy type of the complex of free factors of a free group, Proc. Lond. Math. Soc. (3) 121 (2020), no.~6, 1737--1765. \arxiv{1810.09380}

\bibitem{Bykovskii}
V. A. Bykovski\u{\i}, Generating elements of the annihilating ideal for modular symbols, Funct. Anal. Appl. 37 (2003), no.~4, 263--272; translated from Funktsional. Anal. i Prilozhen. 37 (2003), no. 4, 27--38, 95.

\bibitem{ChurchFarbPutmanVanish}
T. Church, B. Farb, and\ A. Putman, The rational cohomology of the mapping class group vanishes in its virtual cohomological dimension, Int. Math. Res. Not. IMRN 2012, no.~21, 5025--5030. \arxiv{1108.0622}

\bibitem{ChurchFarbPutmanGL}
T. Church, B. Farb, and\ A. Putman, Integrality in the Steinberg module and the top-dimensional cohomology of $\SL_n \cO_K$, Amer. J. Math. 141 (2019), no.~5, 1375--1419. \arxiv{1501.01307}

\bibitem{ChurchPutmanCodim1}
T. Church\ and\ A. Putman, The codimension-one cohomology of $\SL_n \Z$, Geom. Topol. 21 (2017), no.~2, 999--1032. \arxiv{1507.06306}

\bibitem{CostaBases}
I. Sadofschi Costa, The complex of partial bases of a free group, Bull. Lond. Math. Soc. 52 (2020), no.~1, 109--120. \arxiv{1711.09954}

\bibitem{DeligneRF}
P. Deligne, Extensions centrales non r\'{e}siduellement finies de groupes arithm\'{e}tiques, C. R. Acad. Sci. Paris S\'{e}r. A-B 287 (1978), no.~4, {\rm A}203--{\rm A}208.

\bibitem{FarbMargalitPrimer}
B. Farb\ and\ D. Margalit, {\it A primer on mapping class groups}, Princeton Mathematical Series, 49, Princeton University Press, Princeton, NJ, 2012.

\bibitem{FullartonPutman}
N. J. Fullarton\ and\ A. Putman, The high-dimensional cohomology of the moduli space of curves with level structures, J. Eur. Math. Soc. (JEMS) 22 (2020), no.~4, 1261--1287. \arxiv{1610.03768}

\bibitem{GRW}
S. Galatius\ and\ O. Randal-Williams, Homological stability for moduli spaces of high dimensional manifolds. I,
J. Amer. Math. Soc. 31 (2018), no.~1, 215--264. \arxiv{1403.2334}

\bibitem{GerlitsThesis}
F. Gerlits, Invariants in Chain Complexes of Graphs, 2002 PhD thesis.

\bibitem{GerlitsCode}
F. Gerlits, computer code at \url{http://pi.math.cornell.edu/m/Research/Dissertations/Gerlits/code}.

\bibitem{HarerDuality}
J. L. Harer, The virtual cohomological dimension of the mapping class group of an orientable surface, Invent. Math. 84 (1986), no.~1, 157--176.

\bibitem{HatcherStabilization}
A. Hatcher, Homological stability for automorphism groups of free groups, Comment. Math. Helv. 70 (1995), no.~1, 39--62.

\bibitem{HatcherVogtmann}
A. Hatcher\ and\ K. Vogtmann, The complex of free factors of a free group, Quart. J. Math. Oxford Ser. (2) 49 (1998), no.~196, 459--468.

\bibitem{HatcherVogtmannRevised}
A. Hatcher\ and\ K. Vogtmann, The complex of free factors of a free group, 2022 revised version, \arxiv{2203.15602}

\bibitem{HatcherVogtmannTethers}
A. Hatcher\ and\ K. Vogtmann, Tethers and homology stability for surfaces, Algebr. Geom. Topol. 17 (2017), no.~3, 1871--1916. \arxiv{1508.04334}

\bibitem{HatcherWahlBoundaries}
A. Hatcher\ and\ N. Wahl, Stabilization for the automorphisms of free groups with boundaries, Geom. Topol. 9 (2005), 1295--1336. \arxiv{math/0406277}

\bibitem{HatcherWahl3Manifolds}
A. Hatcher\ and\ N. Wahl, Stabilization for mapping class groups of 3-manifolds, Duke Math. J. 155 (2010), no.~2, 205--269. \arxiv{0709.2173}

\bibitem{Hempel}
J. Hempel, {\it $3$-Manifolds}, Princeton University Press, Princeton, NJ, 1976.

\bibitem{Landgraf}
J. Landgraf, Cutting and pasting in the Torelli subgroup of $\Out(F_n)$, preprint 2021. \arxiv{2106.14147}

\bibitem{LaudenbachPaper}
F. Laudenbach, Sur les $2$-sph\`eres d'une vari\'{e}t\'{e} de dimension $3$, Ann. of Math. (2) 97 (1973), 57--81. 

\bibitem{LaudenbachBook}
F. Laudenbach, {\it Topologie de la dimension trois: homotopie et isotopie}, Ast\'{e}risque, No. 12, Soci\'{e}t\'{e} Math\'{e}matique de France, Paris, 1974.

\bibitem{LeeSzczarba}
R. Lee\ and\ R. H. Szczarba, On the homology and cohomology of congruence subgroups, Invent. Math. 33 (1976), no.~1, 15--53. 

\bibitem{MillerNagpalPatzt}
J. Miller, R. Nagpal, and\ P. Patzt, Stability in the high-dimensional cohomology of congruence subgroups, Compos. Math. 156 (2020), no.~4, 822--861. \arxiv{1806.11131}

\bibitem{MillerPatztPutman}
J. Miller, P. Patzt, and\ A. Putman, On the top-dimensional cohomology groups of congruence subgroups of $\SL(n,\Z)$, Geom. Topol. 25 (2021), no.~2, 999--1058. \arxiv{1909.02661}

\bibitem{MilnorConnectSum}
J. Milnor, A unique decomposition theorem for $3$-manifolds, Amer. J. Math. 84 (1962), 1--7.

\bibitem{QuillenPoset}
D. Quillen, Homotopy properties of the poset of nontrivial $p$-subgroups of a group, Adv. Math. 28 (1978), no.~2, 101--128.

\bibitem{RolfsenKnots}
D. Rolfsen, {\it Knots and links}, Mathematics Lecture Series, No. 7, Publish or Perish, Inc., Berkeley, CA, 1976. 

\bibitem{SolomonTits}
L. Solomon, The Steinberg character of a finite group with $BN$-pair, in {\it Theory of Finite Groups (Symposium, Harvard Univ., Cambridge, Mass., 1968)}, 213--221, Benjamin, New York. 

\bibitem{Spanier}
E. H. Spanier, {\it Algebraic topology}, McGraw-Hill Book Co., New York, 1966.

\bibitem{ZeemanRelative}
E. C. Zeeman, Relative simplicial approximation, Proc. Cambridge Philos. Soc. 60 (1964), 39--43.

\end{thebibliography}
\end{document}